\documentclass[11pt]{amsart}
\newcommand{\sk}{\smallskip}
\newcommand{\mk}{\medskip}

\newcommand{\M}{\boldsymbol{\mathcal M}}
\newcommand{\Res}{\operatorname{Res}}
\newcommand{\Alt}{\operatorname{Alt}}

\newcommand{\dd}{\mathrm d}
\newcommand{\Log}{\operatorname{Log}}

\newcommand{\C}{\mathbf C}
\newcommand{\R}{\mathbf R}

\usepackage[all]{xy}

\usepackage{amsmath}
\usepackage{amssymb}
\usepackage{amsthm}
\usepackage{latexsym}
\usepackage{pstricks}
\usepackage{amsbsy}
\usepackage{amscd}
\usepackage{mathrsfs}
\usepackage{tipa}
\usepackage{txfonts}
\usepackage{amsfonts}
\usepackage{graphicx}
\usepackage{tabularx}
\usepackage{listings}
\usepackage{tikz}
\usepackage{hyperref}

\newtheorem{thm}{Theorem}[section]
\newtheorem{cor}[thm]{Corollary}
\newtheorem{lem}[thm]{Lemma}

\newtheorem{prop}[thm]{Proposition}

\newtheorem{rem}[thm]{Remark}
\newtheorem{defn-prop}[thm]{Definition-Proposition}

\author{Luc Pirio} \address{\href{https://lmv.math.cnrs.fr}{Laboratoire de Math\'ematiques de Versailles}, CNRS \& Universit\'e Paris-Saclay $-$ UVSQ, France\\
E-mail: {\href{mailto:luc.pirio@uvsq.fr}{luc.pirio@uvsq.fr}}}

\title[]%The $(n+3)$-web by rational curves on $\mathcal M_{0,n+3}$.]
 {On Euler's abelian relation of the web induced by\\  the  forgetful maps on the moduli space $\boldsymbol{\mathcal M_{0,n+3}}$:\\  the case when $\boldsymbol{n}$ is even}
 
\begin{document}
%%%%%%%%%%%%%%%%%%%%%%%%%%%%%%%%%%%%%%%%%%%%%%%
%%%%%%%%%%%%%%%%%%%%%%%%%%%%%%%%%%%%%%%%%%%%%%%

\begin{abstract}
In \cite{M06}, for every $n\geq 2$, we studied the curvilinear web $\boldsymbol{\mathcal W}_{0,n+3}$ on the moduli space $\boldsymbol{\mathcal M}_{0,n+3}$, defined by the $n+3$ forgetful maps $
\boldsymbol{\mathcal M}_{0,n+3}\rightarrow \boldsymbol{\mathcal M}_{0,n+2}$, 
with particular emphasis on its abelian relations.
In particular, when $n$ is even, we conjectured an explicit formula for its Euler abelian relation. This relation is constructed \`a la Gelfand--MacPherson from an invariant differential form representing the appropriate power of the Euler class of the tautological bundle 
of the Grassmannian manifold %$G_2^{\mathrm{or}}(\mathbf R^{n+3})$ 
of oriented $2$-planes in $\mathbf R^{n+3}$.
\sk 

%In this note, we obtain an explicit antisymmetrized formula for Euler's abelian
%relation for every even $n\geq 2$, and prove that it is equivalent, up to an explicit nonzero
%scalar, to the formula conjectured in \cite{M06}.

In this note, we obtain an explicit antisymmetrized formula for Euler's abelian
relation for every even $n\geq 2$, and prove that it coincides, up to the explicit factor
$-(n+1)(n+2)/2$, with the formula conjectured in \cite{M06}, which is therefore established
for all even $n$.
\end{abstract}
 \maketitle 

This note is a complement to  \cite{M06}, to which the reader is referred  for background, notation and related material. 
%%%%%%%%%%%%%%%%%%%%%%%%%%%%%%%%%%%%%%%%%%%%%%%
%%%%%%%%%%%%%%%%%%%%%%%%%%%%%%%%%%%%%%%%%%%%%%%
\section{Introduction}
%%%%%%%%%%%%%%%%%%%%%%%%%%%%%%%%%%%%%%%%%%%%%%%
%%%%%%%%%%%%%%%%%%%%%%%%%%%%%%%%%%%%%%%%%%%%%%%

%%%%%%%%%%%%%%%%%%%%%%%%%%%%%%%%%%%%%%%%%%%%%%%
%%%%%%%%%%%%%%%%%%%%%%%%%%%%%%%%%%%%%%%%%%%%%%%
\subsection{Background and motivation}
It is well known that Abel's celebrated five-term relation for the dilogarithm is carried by the $5$-web on $\boldsymbol{\mathcal M}_{0,5}$ induced by the five forgetful maps
\[
\boldsymbol{\mathcal M}_{0,5}\longrightarrow \boldsymbol{\mathcal M}_{0,4}.
\]
It is therefore natural to investigate the abelian relations of the higher-dimensional webs
$\boldsymbol{\mathcal W}_{0,n+3}$ defined, for every integer $n\geq 2$, by the $(n+3)$ forgetful maps
\[
\boldsymbol{\mathcal M}_{0,n+3}\longrightarrow \boldsymbol{\mathcal M}_{0,n+2}.
\]
These webs were studied by Damiano in \cite{DamianoThesis,Damiano}, following a geometric approach initiated by Gelfand and MacPherson \cite{GelfandMacPherson}. 
Gelfand and MacPherson used this
approach  to give a geometric construction of the dilogarithm and, subsequently, a differential-geometric explanation for the fact that it satisfies Abel's identity.

Following the Gelfand--MacPherson approach, Damiano constructed a distinguished $(n-1)$-abelian relation for
$\boldsymbol{\mathcal W}_{0,n+3}$ from an
${\rm SO}\big(\mathbf R^{n+3}\big)$-invariant differential form representing the $n$-th power of the Euler class of the tautological rank-$2$ bundle over the Grassmannian of oriented $2$-planes in $\mathbf R^{n+3}$. This relation is called the \emph{Euler abelian relation} and is denoted by
$\boldsymbol{\mathcal E}_{n}$ (see \cite[\S5]{M06} for more details).
\sk

In \cite{M06}, we revisited Damiano's study of the webs
$\boldsymbol{\mathcal W}_{0,n+3}$ and corrected a non-trivial error arising when $n$ is odd. Among other results, 
for even $n$, in a natural affine coordinate
system on
$\boldsymbol{\mathcal M}_{0,n+3}$, we obtained a conjectural explicit formula $\boldsymbol{{\mathcal E}}_{n}'$ for the components of
$\boldsymbol{\mathcal E}_{n}$. By direct symbolic computations, we verified this formula for every even $n\leq 12$; see formula (5.65) and Proposition 5.9.5 in \cite[\S5.9.2]{M06}.

The purpose of the present text is to establish an explicit closed formula for the Euler abelian relation for arbitrary even $n\geq 2$, and to relate it to the formula conjectured in \cite[\S5.9.2]{M06}.
%\begin{center}
%$\star$
%\end{center}
\mk 

We state our main results below, after introducing the necessary
notation and definitions.
%%%%%%%%%%%%%%%%%%%%%%%%%%%%%%%%%%%%%%%%%%%%%%%
%%%%%%%%%%%%%%%%%%%%%%%%%%%%%%%%%%%%%%%%%%%%%%%

%%%%%%%%%%%%%%%%%%%%%%%%%%%%%%%%%%%%%%%%%%%%%
%%%%%%%%%%%%%%%%%%%%%%%%%%%%%%%%%%%%%%%%%%%%%
%\subsection{}
%%%%%%%%%%%%%%%%%%%%%%%%%%%%%%%%%%%%%%%%%%%%%
%%%%%%%%%%%%%%%%%%%%%%%%%%%%%%%%%%%%%%%%%%%%%

%\cite[\S5.7.1]{PiriodP}
%
%
%\cite[\S5.9]{PiriodP}

%%%%%%%%%%%%%%%%%%%%%%%%%%%%%%%%%%%%%%%%%%%%%
%%%%%%%%%%%%%%%%%%%%%%%%%%%%%%%%%%%%%%%%%%%%%
%\subsection{An equivalent but more intrinsic approach}
%%%%%%%%%%%%%%%%%%%%%%%%%%%%%%%%%%%%%%%%%%%%%
%%%%%%%%%%%%%%%%%%%%%%%%%%%%%%%%%%%%%%%%%%%%%
\subsection{Some notation, definitions and statements of the results} 

\subsubsection{Some notation and conventions.} 
We use the following notation and conventions  for any integer $N\geq 2$: 
%%%%%%%%%%%%%%%%%%%%%%%%%%%%%%%%%%%%%%%%%%%%%
\begin{itemize}
%%%%%%%%%%%%%%%%%%%%%%%%%%%%%%%%%%%%%%%%%%%%%
\item Given a permutation $\sigma$ of $[\hspace{-0.06cm}[N]\hspace{-0.06cm}] =\big\{1,\ldots,N\big\}$, 
\begin{itemize}
\item we sometimes use the notation $\sigma_i=\sigma(i)$ for $i=1,\ldots,N$;
\item its signature will be denoted by $(-1)^\sigma$; 
\item the map $(x_i)_{i=1}^N \mapsto \big(x_{\sigma(i)}\big)_{i=1}^N$ will be denoted by the same symbol $\sigma$. 
\end{itemize}
\sk 
%%%%%%%%%%%%%%%%%%%%%%%%%%%%%%%%%%%%%%%%%%%%%
\item given a differential form $\Psi$ expressed in affine coordinates $x_i$, and 
provided that all the pullbacks below are defined, we set 
\[
\Alt_N\big(\Psi\big)
=
\Alt_{\mathfrak S_N}\big(\Psi\big)
=\frac{1}{N!}\sum_{\sigma\in{\mathfrak S}_N}
(-1)^\sigma\,\sigma^*\Psi\, .
\]
The same convention will be used for antisymmetrization with respect to any finite totally ordered index set, with the obvious modifications. 
\sk
\item 
Given a function $f$ for which this makes sense, we also use the notation
\[
\dd\Log f=\frac{\dd f}{f},
\]
so no choice of a  logarithm branch is needed.
%%%%%%%%%%%%%%%%%%%%%%%%%%%%%%%%%%%%%%%%%%%%%
%%%%%%%%%%%%%%%%%%%%%%%%%%%%%%%%%%%%%%%%%%%%%
\end{itemize}
%%%%%%%%%%%%%%%%%%%%%%%%%%%%%%%%%%%%%%%%%%%%%
%\mk 

%%%%%%%%%%%%%%%%%%%%%%%%%%%%%%%%%%%%%%%%%%%%%
\subsubsection{The spaces.}
%%%%%%%%%%%%%%%%%%%%%%%%%%%%%%%%%%%%%%%%%%%%%
Let $A=\{a_1,\ldots, a_N\}$ be an ordered set of labels of cardinality $N\geq 4$ (with $a_i<a_j$ for every $1\leq i<j\leq N$). We identify $A$ with $[\hspace{-0.06cm}[ N]\hspace{-0.06cm}]$
via the order-preserving bijection $a_i\mapsto i$.
\sk

Let $V$ be a fixed two-dimensional complex vector space, and write  $\mathbf P^1=\mathbf P(V)$.
 We fix once and for all a nonzero alternating form $
\det:\wedge^2 V\longrightarrow\C
$,  and for $(v_a)_{a\in A}\in V^A$ and for $a,a'\in A$, we set 
$$
[a a']={\rm det}\big( v_a,v_{a'})
$$
and if $a=a_i$ and $a'=a_j$, we will also write $[ij]=[a_i a_j]$. 
\sk 

We consider the following spaces 
\begin{itemize}
\item  $
\widetilde U_A
=
\left\{\, 
(v_a)\in\big(V\setminus\{0\}\big)^A\,;\,
[aa']\neq0\ \text{for }a\neq a'
\right\}.
$
\item  $
U_A
=
\left\{
(p_a)\in\big(\mathbf P^1\big)^A\,;\,
p_a\neq p_{a'}\ \text{for }a\neq a'
\right\}=\big(\mathbf P^1\big)^A\setminus D_A$ where $D_A$ stands for the union of the pairwise diagonals; 
\item the moduli space 
$\boldsymbol{\mathcal M}_{0,A}$ 
of projective configurations of distinct points on $\mathbf P^1$ indexed by $A$: 
$$
\boldsymbol{\mathcal M}_{0,A}=\Big( \big(\mathbf P^1\big)^A\setminus D_A\Big)\,{\Big/\,{\bf PGL}(V)} = U_A\,{\big/\,}_{{\bf PGL}(V)} \, . 
$$
\end{itemize}

As is well known, $\boldsymbol{\mathcal M}_{0,A}$ is the quotient of $U_A$ by the natural action of ${\bf PGL}(V)$ 
and $U_A$ is in turn the quotient of $\widetilde U_A$ by the independent scalar rescalings $v_a\mapsto {\lambda}_a v_a$ with ${\lambda}_a\in \mathbf C^*$.

%%%%%%%%%%%%%%%%%%%%%%%%%%%%%%%%%%%%%%%%%%%%%
\subsubsection{The differential form.} 
%%%%%%%%%%%%%%%%%%%%%%%%%%%%%%%%%%%%%%%%%%%%%
We define 
$$
\eta_{A}= {\rm Log}\, \big\lvert\,  [12]\,\big\lvert 
\bigwedge_{s=3}^{N-1}  \mathrm d  {\rm Log} \left( 
\frac{[1s][2N]}{[2s][1N]}
\right) 
$$
and 
we set 
\begin{equation}
\label{Eq:OmegaA-Definition}
\omega_A= {\rm Alt}_{{\mathfrak S}_A}\big( \eta_{A}\big)\, , 
\end{equation}

A priori, $\omega_A$ is a complex-valued real-analytic differential form of type $(N-3,0)$
on $\widetilde U_A$.  

%%%%%%%%%%%%%%%%%%%%%%%%%%%%%%%%%%%%%%%%%%%%%%%
\begin{prop}
\label{Prop:descent}
%%%%%%%%%%%%%%%%%%%%%%%%%%%%%%%%%%%%%%%%
%%%%%%%%%%%%%%%%%%%%%%%%%%%%%%%%%%%%%%%%%%%%%%%
The differential form $\omega_A$:
\begin{enumerate}
\item[{\bf 1.}] descends to  the moduli space $\boldsymbol{\mathcal M}_{0,A}$;
\item[{\bf 2.}]  is nonzero when $N$ is even. 
\end{enumerate} 
\sk 
%%%%%%%%%%%%%%%%%%%%%%%%%%%%%%%%%%%%%%%%%%%%%%%
\end{prop}%%%%%%%%%%%%%%%%%%%%%%%%%%%%%%%%%%%%%%%%
%%%%%%%%%%%%%%%%%%%%%%%%%%%%%%%%%%%%%%%%%%%%%%%
\noindent {\bf Note:} %When $N$ is odd, it can be verified that $\omega_A$ vanishes identically. 
When $N$ is odd, $\omega_A$ vanishes identically: indeed, $\eta_A$
is invariant under the transposition exchanging its first two labels,
while this transposition is odd.

%%%%%%%%%%%%%%%%%%%%%%%%%%%%%%%%%%%%%%%%%%%%%
\subsubsection{The web $\boldsymbol{\mathcal W}_{0,A}$ and its abelian relations.}
%%%%%%%%%%%%%%%%%%%%%%%%%%%%%%%%%%%%%%%%%%%%%
We fix an even integer $N\geq 4$ and assume throughout this subsection that
$
A=[\hspace{-0.06cm}[N+1]\hspace{-0.06cm}]
=\big\{1,\ldots,N+1\big\}$.\sk 

For every $a\in A$, we set
\[
A_a=A\setminus\{a\}
=
\big\{\,1,\ldots,\widehat{a},\ldots,N+1\,\big\},
\]
endowed with its natural ordering, and denote by
\[
f_a:\boldsymbol{\mathcal M}_{0,A}
\longrightarrow
\boldsymbol{\mathcal M}_{0,A_a}
\]
the forgetful map that forgets the point labelled by $a$.
By definition, the web $\boldsymbol{\mathcal W}_{0,A}$ is the
$(N+1)$-web by rational curves on $\boldsymbol{\mathcal M}_{0,A}$
defined by the forgetful maps $f_a$, namely
\[
\boldsymbol{\mathcal W}_{0,A}
=
\boldsymbol{\mathcal W}
\big(\, f_a \,\big\lvert\, a\in A\,\big).
\]

Recall that a \emph{$k$-abelian relation} for this web on an open subset
$D\subset \boldsymbol{\mathcal M}_{0,A}$ is an $A$-tuple
$(\eta_a)_{a\in A}$, 
where, for each $a\in A$, $\eta_a$ is a differential $k$-form on
$f_a(D)\subset \boldsymbol{\mathcal M}_{0,A_a}$, such that the following
fundamental differential identity holds:
\[
\sum_{a\in A} f_a^*(\eta_a)=0.
\]
We shall often identify an abelian relation $(\eta_a)_{a\in A}$ with the corresponding tuple of
pulled-back forms $\big(f_a^*(\eta_a)\big)_{a\in A}$.

%%%%%%%%%%%%%%%%%%%%%%%%%%%%%%%%%%%%%%%%%%%%%
\subsubsection{The real setting.}
%%%%%%%%%%%%%%%%%%%%%%%%%%%%%%%%%%%%%%%%%%%%%
All the objects considered above, namely configuration spaces,
moduli spaces, forgetful maps, webs, admit real versions. We will use a superscript ${}^{\mathbb R}$ to mean that we are considering the real versions. For instance, $\boldsymbol{\mathcal M}_{0,A}^{\mathbb R}$ stands for the real locus of the complex moduli space $\boldsymbol{\mathcal M}_{0,A}$, the forgetful maps $f_a$ are defined over $\mathbb R$ and therefore admit real versions 
$\boldsymbol{\mathcal M}_{0,A}^{\mathbb R}\rightarrow \boldsymbol{\mathcal M}_{0,A_a}^{\mathbb R}$, again denoted by $f_a$,  which are the first integrals of the real version $\boldsymbol{\mathcal W}_{0,A}^{\mathbb R}$ of the complex web $\boldsymbol{\mathcal W}_{0,A}$. 

In some cases, to simplify notation, the superscript ${}^{\mathbb R}$ will be dropped when there is no ambiguity regarding the setting.  We call the
connected components of $\boldsymbol{\mathcal M}_{0,A}^{\mathbb R}$
its \emph{chambers}. The forgetful maps and the action of
$\mathfrak S_A$ are defined on the whole real moduli space.

The functions $\Log|[ij]|$ are single-valued and real analytic
on the corresponding spaces of lifts. By
Proposition~\ref{Prop:descent}, their antisymmetrized combinations
defining the forms $\omega_{A_a}$ descend to globally defined
real-analytic forms on $\boldsymbol{\mathcal M}_{0,A_a}$.
Thus each $\omega_{A_a}$ is defined on the whole target of $f_a$,
including all its chambers.

Accordingly, the abelian relation constructed below is global
on the whole real moduli space, with components pulled back
from globally defined forms on the respective targets.
Our comparison with Euler's abelian relation obtained in
\cite{M06} will take place in this real-analytic setting.
%\medskip 

%%%%%%%%%%%%%%%%%%%%%%%%%%%%%%%%%%%%%%%%%%%%%
\subsubsection{Main results: differential identity and equivariance.}
%%%%%%%%%%%%%%%%%%%%%%%%%%%%%%%%%%%%%%%%%%%%%
With the notation of the previous subsection, for every $a\in A$ we set 
\[
\omega_a=\omega_{A_a} \, . 
\]

Thus, setting $n=N-2$, the web $\boldsymbol{\mathcal W}_{0,A}$ is naturally identified
with $\boldsymbol{\mathcal W}_{0,n+3}$.

%%%%%%%%%%%%%%%%%%%%%%%%%%%%%%%%%%%%%%%%%%%%%
\begin{thm}
\label{Prop:Differential-Identity} 
%%%%%%%%%%%%%%%%%%%%%%%%%%%%%%%%%%%%%%%
%%%%%%%%%%%%%%%%%%%%%%%%%%%%%%%%%%%%%%%%%%%%%
{\bf 1.} The  forms $\omega_a$ satisfy the differential relation
%%%%%%%%%%%%%%%%%%%%%%%%%%%%%%%%%%%%%%%%%%%%%
\begin{equation}
\label{Eq:sum-(-1)a-omega-a=0}
\sum_{a=1}^{n+3}
(-1)^a f_a^*(\omega_a)=0
\end{equation}
on $\boldsymbol{\mathcal M}_{0,A}$. 
%%%%%%%%%%%%%%%%%%%%%%%%%%%%%%%%%%%%%%%%%%%%%
 Consequently  
%$\boldsymbol{E}_N=\Big(  (-1)^a f_a^*(\omega_a)\Big)_{a\in A}$ can be seen as a global abelian relation for $\boldsymbol{\mathcal W}_{0,A}$
%the $(N+1)$-tuple 
$$\boldsymbol{E}_{n}=\Big(  (-1)^a f_a^*(\omega_a)\Big)_{a\in A}$$ defines a global real-analytic $(n-1)$-abelian relation for $\boldsymbol{\mathcal W}_{0,A}$. 
%hence   
%$\boldsymbol{E}_N=\Big(  (-1)^a f_a^*(\omega_a)\Big)_{a\in A}$ can be seen as a %global abelian relation for $\boldsymbol{\mathcal W}_{0,A}$. 
%%%%%%%%%%%%%%%%%%%%%%%%%%%%%%%%%%%%%%%%%%%%%
\sk 

\noindent 
{\bf 2.} For the natural action of $\mathfrak S_A$ on tuples $R=(R_a)_{a\in A}$ of differential forms given by $(\sigma\cdot R)_a=\sigma^*\big( R_{\sigma^{-1}(a)}\big)$ for $a\in A$ and $\sigma \in \mathfrak S_A$, we have 
\begin{equation}
\label{Eq:signature-invariance}
\sigma\cdot \boldsymbol{E}_n=(-1)^\sigma \boldsymbol{E}_n\qquad \forall\, \sigma\in \mathfrak S_A \, .
%{\bf 2.} For any $\sigma\in {\mathfrak S}_A$, we have $\sigma^*\big( f_a^*(\omega_a)\big)  %=(-1)^\sigma f_{\sigma(a)}^*(\omega_{\sigma(a)})$ for every $a\in A$, that is 
%\begin{equation}
%\label{Eq:signature-invariance}
%\sigma^* \boldsymbol{E}_N=(-1)^\sigma \boldsymbol{E}_N\, .
\end{equation}
%%%%%%%%%%%%%%%%%%%%%%%%%%%%%%%%%%%%%%%%%%%%%
\end{thm}%%%%%%%%%%%%%%%%%%%%%%%%%%%%%%%%%%%%%%%
%%%%%%%%%%%%%%%%%%%%%%%%%%%%%%%%%%%%%%%%%%%%%

In Section~\ref{S:Proof-Coro-Main}, we prove that the relation
constructed above is a nonzero constant multiple of the
geometric Euler relation, establishing the following result.
%It is known that, up to a nonzero scalar, 
%Euler's abelian relation is the unique abelian relation of
%$\boldsymbol{\mathcal W}_{0,n+3}^{\mathbb R}$ which transforms according to the sign %character under the natural action of
%$\mathfrak S_{n+3}$ (see  \cite[Cor\,5.7.4]{M06}). Therefore, by %Theorem~\ref{Prop:Differential-Identity} and Proposition~\ref{Prop:descent}, the abelian %relation constructed above is, up to a nonzero scalar factor, Euler's abelian relation.
%%%%%%%%%%%%%%%%%%%%%%%%%%%%%%%%%%%%%%%%%%%%%
\begin{thm}%%%%%%%%%%%%%%%%%%%%%%%%%%%%%%%%%%%%%%%
%%%%%%%%%%%%%%%%%%%%%%%%%%%%%%%%%%%%%%%%%%%%%
\label{Coro:main}
By restriction to the real locus, 
$\boldsymbol{E}_{n}$ gives a global real-analytic abelian relation of top degree for $\boldsymbol{\mathcal W}_{0,n+3}^{\mathbb R}$, which coincides with Euler's relation up to multiplication by a nonzero constant. 
%%%%%%%%%%%%%%%%%%%%%%%%%%%%%%%%%%%%%%%%%%%%%
\end{thm}%%%%%%%%%%%%%%%%%%%%%%%%%%%%%%%%%%%%%%%%
%%%%%%%%%%%%%%%%%%%%%%%%%%%%%%%%%%%%%%%%%%%%%

%\newpage 

In \cite[\S5.9.2]{M06}, we gave a conjectural explicit formula for the components of a
nonzero multiple $\boldsymbol{\mathcal E}'_n$ of Euler's abelian relation, namely the abelian
relation \cite[(5.64)]{M06}, and, by symbolic computations, we verified it for every even
$n\leq 12$. Here we prove this for every even $n\geq 2$; more precisely, we prove the following.
%%%%%%%%%%%%%%%%%%%%%%%%%%%%%%%%%%%%%%%%%%%%%
\begin{cor}
\label{Cor:comparison}
As abelian relations for $\boldsymbol{\mathcal W}^{\mathbb R}_{0,A}$, $\boldsymbol{E}_{n}$ and
$\boldsymbol{\mathcal E}'_{n}$ coincide up to multiplication by a nonzero scalar. More precisely,
$$
\boldsymbol{\mathcal E}'_{n}=-\frac{(n+1)(n+2)}{2}\,\boldsymbol{E}_{n}\, .
$$
\end{cor}
\vspace{0.1cm}

%In \cite[\S5.9.2]{M06}, we gave a conjectural explicit formula 
%$\boldsymbol{{\mathcal E}}_n'$
%for a nonzero multiple $\boldsymbol{\mathcal E}_n$ of Euler's abelian relation and by formal %computations, we proved that this formula indeed holds true for all even integers $n\leq 12$. 
%
%Here we prove that for any $n\geq 2$ even, we have 
%%%%%%%%%%%%%%%%%%%%%%%%%%%%%%%%%%%%%%%%%%%%%
%\begin{cor} 
%\label{Cor:comparison}
%%%%%%%%%%%%%%%%%%%%%%%%%%%%%%%%%%%%%%%
%%%%%%%%%%%%%%%%%%%%%%%%%%%%%%%%%%%%%%%%%%%%%
 %As abelian relations for $\boldsymbol{\mathcal W}^{\mathbb R}_{0,A}$, $\boldsymbol{E}%_{n}$ and $\boldsymbol{\mathcal E}_{n}'$ coincide up to multiplication by a nonzero scalar. 
 %Actually, we have 
%$$
%{}^{} \qquad  \boldsymbol{E}_{n}= \frac{2}{(n+1)(n+2)} \, 
%\boldsymbol{\mathcal E}_{n}'\, . 
%\qquad (sign\,?)
%$$
%\smallskip 
%%%%%%%%%%%%%%%%%%%%%%%%%%%%%%%%%%%%%%%%%%%%%
%\end{cor} %%%%%%%%%%%%%%%%%%%%%%%%%%%%%%%%%%%%%%%%
%%%%%%%%%%%%%%%%%%%%%%%%%%%%%%%%%%%%%%%%%%%%%
%\vspace{0.1cm} 

%%%%%%%%%%%%%%%%%%%%%%%%%%%%%%%%%%%%%%%%%%%%%
\begin{rem}%%%%%%%%%%%%%%%%%%%%%%%%%%%%%%%%%%%%%%%
%%%%%%%%%%%%%%%%%%%%%%%%%%%%%%%%%%%%%%%%%%%%%
Dropping the absolute value in the definition of $\eta_A$ and defining
$\omega_A$ as in \eqref{Eq:OmegaA-Definition}  would give a holomorphic
but potentially multivalued differential form on the corresponding
complex moduli space.
Working locally, or on suitable domains, and choosing compatible
branches of the logarithms appearing in the definition of $\omega_a$
on $\boldsymbol{\mathcal M}_{0,A_a}$ for each $a\in A$, one could obtain
a tuple $\big((-1)^a f_a^*(\omega_a)\big)_{a\in A}$ defining a
holomorphic, potentially multivalued $(n-1)$-abelian relation for the
complex web $\boldsymbol{\mathcal W}_{0,A}$.
For simplicity, and to save space, we have decided not to pursue this
in the present paper, as our main interest is the connection with
Euler's abelian relation studied in
\cite{M06}, which is real by construction.
\sk

Combined with Theorem~\ref{Coro:main}, we expect this approach to show that, for every positive
even $n$, the complexification of Euler's abelian relation extends to a global multivalued abelian
relation on $\boldsymbol{\mathcal M}_{0,n+3}$ with unipotent monodromy (analytic continuation only
changes the logarithms by constants in $2i\pi\mathbf Z$, hence the relation by a rational, i.e.\
combinatorial, abelian relation), as conjectured in the first part of \cite[\S5.7.1]{M06}.
%%%%%%%%%%%%%%%%%%%%%%%%%%%%%%%%%%%%%%%%%%%%%
\end{rem}%%%%%%%%%%%%%%%%%%%%%%%%%%%%%%%%%%%%%%%%
%%%%%%%%%%%%%%%%%%%%%%%%%%%%%%%%%%%%%%%%%%%%%

%%%%%%%%%%%%%%%%%%%%%%%%%%%%%%%%%%%%%%%%%%%%%
\subsubsection{Organization of the paper.}
%%%%%%%%%%%%%%%%%%%%%%%%%%%%%%%%%%%%%%%%%%%%%

Section \ref{S:Proof-Prop-1} proves the first point of Proposition~\ref{Prop:descent}, namely the descent of $\omega_A$ to $\boldsymbol{\mathcal M}_{0,A}$; the
second point is proved in Section~\ref{S:Link-explicit} (Corollary \ref{Cor:Second-Point}). Section  \ref{S:Proof-of-Theorem} contains two proofs of identity \eqref{Eq:sum-(-1)a-omega-a=0},
a conceptual antisymmetrization argument and a symbolic one, together with the proof of the
equivariance property. 

In Section \ref{S:Proof-Coro-Main}, we show that $\boldsymbol{E}_n$ spans the sign-isotypic part of the space
of global real abelian relations of 
$\boldsymbol{\mathcal W}_{0,A}^{\mathbb R}$, and we deduce Theorem \ref{Coro:main}. 
\sk

Finally, in Section~\ref{S:Link-explicit}, we 
compare the antisymmetrized expression obtained here with the explicit formula conjectured in \cite[\S5.9.2]{M06} and prove Corollary \ref{Cor:comparison}.

\mk

%%%%%%%%%%%%%%%%%%%%%%%%%%%%%%%%%%%%%%%%%%%%%
\paragraph{\bf Acknowledgement.}%%%%%%%%%%%%%%%%%%%%%%%%%%
%%%%%%%%%%%%%%%%%%%%%%%%%%%%%%%%%%%%%%%%%%%%%
This work was carried out during a six-month stay in Tokyo as a member of the
\href{https://fj-lmi.cnrs.fr}{FJ-LMI} (IRL 2025), an International Research Laboratory jointly operated by CNRS and The University of Tokyo. The author is very grateful for the 
excellent working conditions enjoyed during this stay.
%%%%%%%%%%%%%%%%%%%%%%%%%%%%%%%%%%%%%%%%%%%%%
%%%%%%%%%%%%%%%%%%%%%%%%%%%%%%%%%%%%%%%%%%%%%
%%%%%%%%%%%%%%%%%%%%%%%%%%%%%%%%%%%%%%%%%%%%%
\mk

%%%%%%%%%%%%%%%%%%%%%%%%%%%%%%%%%%%%%%%%%%%%%
\paragraph{\bf Note on the use of AI}
%%%%%%%%%%%%%%%%%%%%%%%%%%%%%%%%%%%%%%%%%%%%%
We briefly describe the role of AI in the genesis of this work. Motivated by the formulas obtained in \cite{PiriodP,CP,WGM} and by the conjectural formula in \cite{M06}, we used ChatGPT 5.6 to explore possible antisymmetrized expressions for Euler's abelian relation and to search for proofs of the corresponding differential identity and equivariance property.

The resulting arguments were subsequently checked, corrected, completed, and rewritten in rigorous mathematical form by the author. AI was used again at a later stage to assist with the finalization and proofreading of the text.

%\newpage
%%%%%%%%%%%%%%%%%%%%%%%%%%%%%%%%%%%%%%%%%%%%%%%
%%%%%%%%%%%%%%%%%%%%%%%%%%%%%%%%%%%%%%%%%%%%%%%
\section{\bf Descent of $\omega_A$ to $\mathcal M_{0,A}$}
\label{S:Proof-Prop-1}
%%%%%%%%%%%%%%%%%%%%%%%%%%%%%%%%%%%%%%%%%%%%%%%
%%%%%%%%%%%%%%%%%%%%%%%%%%%%%%%%%%%%%%%%%%%%%%%
When $N$ is odd, $\omega_A\equiv 0$ (see the Note after Proposition~\ref{Prop:descent}), so there is nothing to prove.
We therefore assume throughout this section that $N\geq 4$ is even, identify $A$ with $\{1,\ldots,N\}$, and put $d=N-3$, which is odd. 

For any $\sigma\in{\mathfrak S}_N$, we set
%\[
%\beta_\sigma=
%\bigwedge_{s=3}^{N-1}
%\dd\Log\left(
%\frac{
%[\sigma(1)\sigma(s)]\,[\sigma(2)\sigma(N)]
%}{[\sigma(2)\sigma(s)]\,[\sigma(1)\sigma(N)]
%}\right).
%\]
\begin{equation}
\label{Eq:b-sigma}
\beta_\sigma
=
\bigwedge_{s=3}^{N-1}
\dd\Log\left(
\frac{
[\sigma_1\sigma_s]\,[\sigma_2\sigma_N]
}{
[\sigma_2\sigma_s]\,[\sigma_1\sigma_N]
}
\right).
\end{equation}

Since $\beta_\sigma= \sigma^*\big( \beta_{{\bf Id}}\big) $ for any $\sigma$, we have 
\begin{equation}
\label{Eq:omega-A}
\omega_A
=
\frac1{N!}
\sum_{\sigma\in{\mathfrak S}_N}
(-1)^\sigma\,
\Log\,\big\lvert\,[\sigma_1\sigma_2]\, \big\rvert\,
\beta_\sigma.
%\tag{1}
\end{equation}
%%%%%%%%%%%%%%%%%%%%%%%%%%%%%%%%%%%%%%%%%%
\subsection{\bf Proof of the first point of Proposition \ref{Prop:descent}}
%%%%%%%%%%%%%%%%%%%%%%%%%%%%%%%%%%%%%%%%%%
We aim to prove that $\omega_A$ is basic for the successive quotient maps $\widetilde U_A\rightarrow U_A\rightarrow \boldsymbol{\mathcal M}_{0,A}$ introduced above.
\sk

First, note that for any four distinct labels $i,j,k,l\in A$, the function
\[
R_{ijkl}
=
\frac{[ik][jl]}{[jk][il]}
\]
is a cross-ratio, hence an invertible regular function on
$\boldsymbol{\mathcal M}_{0,A}$. Consequently,  the differential form 
$
\dd\Log R_{ijkl}
= {\dd R_{ijkl}}/{R_{ijkl}}
$
is intrinsically defined on $\boldsymbol{\mathcal M}_{0,A}$.

In particular, for every $\sigma\in\mathfrak S_N$, the form
\eqref{Eq:b-sigma} is the pull-back of a regular rational
differential $(N-3)$-form on $\M_{0,A}$.
Hence each $\beta_\sigma$ is horizontal with respect to the
successive group actions defining the quotients 
$\widetilde U_A\to U_A\to\M_{0,A}$.
Thus, in order to prove that $\omega_A$ is basic, it remains only
to check its invariance under these actions. The only possible
obstruction to this invariance comes from the logarithmic
coefficients $\Log|[\sigma_1\sigma_2]|$, for $\sigma\in\mathfrak S_N$. 
%In particular, for every $\sigma\in{\mathfrak S}_N$, the form \eqref{Eq:b-sigma}
%is the pull-back of a regular rational differential $(N-3)$-form on
%$\boldsymbol{\mathcal M}_{0,A}$. Thus the only possible obstruction
%to the descent of 
%$\omega_A$  comes from the logarithmic coefficients
%$\Log\big\lvert[\sigma_1\sigma_2]\big\lvert$, for $\sigma\in{\mathfrak S}_N$. 
\medskip

Given $\lambda=({\lambda}_a)_{a\in A}\in (\mathbf C^*)^A $ arbitrary, we 
consider the rescaling $$R_\lambda: \widetilde U_A
\longrightarrow \widetilde U_A: \, 
v=(v_a)_{a\in A}\longmapsto v'=\big({\lambda}_a v_a \big)_{a\in A}\, , $$
 and below we set  $\omega_A'=R_\lambda^*\big( \omega_A\big)$. 
%Given $v=(v_a)_a\in \widetilde U_A$ and $({\lambda}_a)_{a\in A}\in (\mathbf C^*)^A $ arbitrary, we %set $v'=\big({\lambda}_av_a \big)_{a\in A}$ and below we denote by $\omega_A$ and $\omega_A'$ %the differential forms defined by formula \eqref{Eq:omega-A} at $v$ and $v'$ respectively. 
 Let us prove that $\omega_A=\omega_A'$, which will show that $\omega_A$ descends to $U_A$.  
  We set $ \ell_a={\rm Log}\,\lvert \, {\lambda}_a\,\lvert $ for every $a\in A$. 
Since
$
[ij]'={\rm det}(v_i',v_j') =
{\lambda}_i{\lambda}_j [ij]$ for every $i,j=1,\ldots,N$, 
all the cross-ratios occurring in the $\beta_\sigma$'s remain
unchanged, whereas
\[
\Log\, \big\lvert\, [ij]'\, \big\rvert
=
\Log\,\big\lvert\, [ij]\,\big\rvert+\ell_i+\ell_j.
\]
It follows that 
%$
%\omega_A-\omega_A' 
%= 
%(1/{N!})
%\sum_{\sigma\in{\mathfrak S}_N}
% (-1)^\sigma
%\bigl(
%\ell_{\sigma(1)}+\ell_{\sigma(2)}
%\bigr)\, 
%\beta_\sigma
%$.
\[
\omega'_A-\omega_A 
= \frac1{N!}
\sum_{\sigma\in{\mathfrak S}_N}
 (-1)^\sigma
\bigl(
\ell_{\sigma(1)}+\ell_{\sigma(2)}
\bigr)\, 
\beta_\sigma.
\]

We are therefore reduced to proving that 
for arbitrary scalars $c_1,\ldots,c_N$, we have 
%%%%%%%%%%%%%%%%%%%%%%%%%%%%%%%%%%%%%%%
\begin{equation}%%%%%%%%%%%%%%%%%%%%%%%%%%%%%%
%%%%%%%%%%%%%%%%%%%%%%%%%%%%%%%%%%%%%%%
\label{Eq:MRP}
\sum_{\sigma\in{\mathfrak S}_N}
 (-1)^\sigma 
\big(
c_{\sigma_1}+c_{\sigma_2}
\big) \, 
\beta_\sigma
=
0.
%%%%%%%%%%%%%%%%%%%%%%%%%%%%%%%%%%%%%%%
\end{equation}%%%%%%%%%%%%%%%%%%%%%%%%%%%%%%%
%%%%%%%%%%%%%%%%%%%%%%%%%%%%%%%%%%%%%%%
Since the left-hand side is linear in the $c_i$'s, it suffices to
prove that, for every fixed label $a\in A$, we have 
\[
S_a
=
\sum_{ \sigma\in \Sigma_a}
%{\substack{\sigma\in{\mathfrak S}_N\\
%a\in\{\sigma_1,\sigma_2\}}}
(-1)^\sigma  
\beta_\sigma
=
0\, , 
%\tag{4}
\]
where we sum over the set $\Sigma_a\subset \mathfrak S_N$ of permutations $\sigma $ such that $
a\in\{\sigma_1,\sigma_2\}$.  
By relabelling, it is enough to establish that $S_1=0$. 
\sk

%We have
%\[
%\Sigma_1=\Sigma_{11}\sqcup\Sigma_{12},
%\qquad
%\Sigma_{1k}
%=
%\{\sigma\in\Sigma_1\mid \sigma(k)=1\},
%\quad k=1,2.
%\]
We have $\Sigma_1=\Sigma_{11}\sqcup \Sigma_{12}$ where $\Sigma_{1k}$ stands for the set of $\sigma\in \Sigma_1$ such that $\sigma(k)=1$. 
Observe that interchanging the first two entries of a
permutation inverts every cross-ratio occurring in any $\beta_\sigma$.
Hence each logarithmic differential changes sign, and, since $d$ is odd, we have 
%%%%%%%%%%%%%%%%%%%%%%%%%%%%%%%%%%%%%%%%%%%
\[%%%%%%%%%%%%%%%%%%%%%%%%%%%%%%%%%%%%%%%%%%
%%%%%%%%%%%%%%%%%%%%%%%%%%%%%%%%%%%%%%%%%%%
\beta_{\sigma\circ(12)}
=
(-1)^d\beta_\sigma
=
-\beta_\sigma \, .
\]%%%%%%%%%%%%%%%%%%%%%%%%%%%%%%%%%%%%%%%%%%
%%%%%%%%%%%%%%%%%%%%%%%%%%%%%%%%%%%%%%%%%%%
At the same time,  
the signature of $\sigma\circ (12)$ is the opposite of that of $\sigma$, 
%$(-1)^{(12)\circ\sigma}=-(-1)^\sigma$, 
  thus the signed contribution
$(-1)^\sigma\, \beta_\sigma
$ 
is unchanged when the labels occupying the first two positions
are interchanged. In particular, the contributions with the
label $1$ in the first position and those with the label $1$ in
the second position coincide:  we have 
$$
\sum_{ \sigma\in \Sigma_{11}}
(-1)^\sigma  
\beta_\sigma
=\sum_{ \sigma\in \Sigma_{12}}
(-1)^\sigma  
\beta_\sigma
\qquad \mbox{ hence } \qquad S_1= 2\sum_{ \sigma\in \Sigma_{11}}
(-1)^\sigma  
\beta_\sigma\, . 
$$

Put
\[
B=\{2,\ldots,N\},
\qquad
\mbox{ and } 
\qquad 
d=N-3\, .
\]
Now fix distinct $b,b'\in B$, and write the remaining labels in
increasing order as
\[
B\setminus\{b,b'\}
=
\{b_1<\cdots<b_d\}.
\]
Define
\[
\beta_{b b'}
=
\bigwedge_{\nu=1}^{d}
\dd\Log
\left(
\frac{[1b_\nu][bb']}
     {[bb_\nu][1b']}
\right).
\]
We denote by $\sigma_{bb'}\in{\mathfrak S}_N$ the unique 
 permutation such that 
 $$
 \bigl(
\sigma_{bb'}(i)
\bigr)_{i=1}^N = 
 \big(1,b,b_1,\ldots,b_d,b'\big)\, .
 $$
 We have $\sigma_{bb'}\in \Sigma_{11}$ and $\beta_{b b'}=\beta_{\sigma_{bb'}}$. Denote by 
 $\Sigma_{11}^{bb'}$ the set of permutations $\sigma\in \Sigma_{11}$ such that $\sigma(2)=b$ and $\sigma(N)=b'$. Then  $\Sigma_{11}^{bb'}= {\mathfrak S}_{\{b_1,\ldots, b_d\}}\cdot \sigma_{bb'}$. An element $\sigma\in  \Sigma_{11}^{bb'}$ corresponds to a certain ordering of the labels $b_1,\ldots,b_d$ and 
 if the order of two such labels is interchanged, the signature
of the permutation changes sign, and the corresponding wedge
product changes sign as well. We thus have $(-1)^\sigma \beta_\sigma=(-1)^{\sigma_{bb'}} \beta_{bb'}$  for any such $\sigma$, from which it follows that 
\begin{equation}
\label{Eq:S1:2}
S_1 = 2d!
\sum_{\substack{b,b'\in B\\ b\neq b'}}
(-1)^{\sigma_{bb'}}\,\beta_{bb'}.
\end{equation}

%such that
%\[
%\sigma_{bb'}(1)=1,
%\quad
%\sigma_{bb'}(2)=b,
%\quad 
%\sigma_{b,b'}(N)=b' 
%\qquad \mbox{and} 
%\qquad 
%\sigma_{bb'}(2+\nu)=j_\nu
%\qquad\mbox{ for } \nu=1,\ldots,d\, .\]

We next fix $b\in B$, write
$$
B\setminus\{b\}
=
\big\{\,\kappa_1<\cdots<\kappa_{d+1}\, \big\}\, , $$ and 
set
%%%%%%%%%%%%%%%%%%%%%%%%%%%%%%%%%%%%%%%
$$
\Delta_b
= \wedge_{s=2}^{d+1} 
(u_s-u_1) 
\qquad 
\mbox{ with } 
\qquad 
u_s
=
\dd\Log\left(
\frac{[1\kappa_s]}{[b\kappa_s]}
\right)\quad \mbox{for } s=1, \ldots,d+1\,. 
$$
%%%%%%%%%%%%%%%%%%%%%%%%%%%%%%%%%%%%%%%

If $b'=\kappa_r$, then $\beta_{bb'}
=
\bigwedge_{\substack{s=1\\s\neq r}}^{d+1}
(u_s-u_r)$, 
%\[
%\beta_{bb'}
%=
%\bigwedge_{\substack{s=1\\s\neq r}}^{d+1}
%(u_s-u_r),
%\]
where the factors are ordered according to increasing $s$.
We claim that
%%%%%%%%%%%%%%%%%%%%%%%%%%%%%%%%%%%%%%%%%%%%%%
\begin{equation}%%%%%%%%%%%%%%%%%%%%%%%%%%%%%%%%%%%%%
\label{Eq:claim-beta-bb'}
%%%%%%%%%%%%%%%%%%%%%%%%%%%%%%%%%%%%%%%%%%%%%%
\beta_{bb'}=\beta_{b\kappa_r}
=
(-1)^{r-1}\Delta_b.
%%%%%%%%%%%%%%%%%%%%%%%%%%%%%%%%%%%%%%%%%%%%%%
\end{equation}%%%%%%%%%%%%%%%%%%%%%%%%%%%%%%%%%%%%%
%%%%%%%%%%%%%%%%%%%%%%%%%%%%%%%%%%%%%%%%%%%%%%
For $r=1$ this is immediate. If $r>1$, we put $w_s=u_s-u_1$ for $s=2,\ldots,d+1$.
%\[
%w_q=u_q-u_1,
%\qquad q=2,\ldots,d+1.
%\]
Then
$
u_1-u_r=-w_r$ and 
$u_s-u_r=w_s-w_r$ for $s \in\{2,\ldots,d+1\}\setminus \{ r\}$. 
Thus
\[
\beta_{b\kappa_r}
=
(-w_r)
\wedge
(w_2-w_r)
\wedge\cdots\wedge
(w_{r-1}-w_r)
\wedge
(w_{r+1}-w_r)
\wedge\cdots\wedge
(w_{d+1}-w_r).
\]
Every term in the expansion in which one chooses a second copy of
$w_r$ vanishes. Hence
\[
\beta_{b\kappa_r}
=
-w_r\wedge
w_2\wedge\cdots\wedge w_{r-1}
\wedge
w_{r+1}\wedge\cdots\wedge w_{d+1}.
\]
Moving $w_r$ from the first position to its natural position
requires $r-2$ transpositions, and therefore
\[
\beta_{b\kappa_r}
=
(-1)^{r-1}
w_2\wedge\cdots\wedge w_{d+1}
=
(-1)^{r-1}\Delta_b.
\]

It remains to determine the sign
$(-1)^{\sigma_{b\kappa_r}}$. By definition,
$\sigma_{b\kappa_r}$ is represented in one-line notation by
\[
\big(\, 1,b,\kappa_1,\ldots,\widehat{\kappa_r},\ldots,\kappa_{d+1},\kappa_r\,\big).
\]
Since the initial $1$ contributes no inversions, we may compute
the signature inside the ordered set $B=\{2,\ldots,N\}$. Starting from the increasing list
$(2,3,\ldots,N)$,  placing $b$ in the first position of this list requires $b-2$
adjacent transpositions. Once $b$ has been removed, the element
$\kappa_r$ occupies the $r$-th position in the increasing list
$
(\kappa_1,\ldots,\kappa_{d+1})$, 
and moving it to the last position requires $
(d+1)-r=N-2-r$ 
adjacent transpositions. Therefore
%\[
%(-1)^{\sigma_{b\kappa_r}}
%=
%(-1)^{(b-2)+(N-2-r)}=(-1)^{b-r}\, , 
%\]
%
%%%%%%%%%%%%%%%%%%%%%%%%%%%%%%%%%%%%%%%%
\begin{equation}%%%%%%%%%%%%%%%%%%%%%%%%%%%%%%%
\label{Eq:signat}%%%%%%%%%%%%%%%%%%%%%%%%%%%%%%%
%%%%%%%%%%%%%%%%%%%%%%%%%%%%%%%%%%%%%%%%
(-1)^{\sigma_{b\kappa_r}}
=
(-1)^{(b-2)+(N-2-r)}=(-1)^{b-r}\, , 
%%%%%%%%%%%%%%%%%%%%%%%%%%%%%%%%%%%%%%%%
\end{equation}%%%%%%%%%%%%%%%%%%%%%%%%%%%%%%%%
%%%%%%%%%%%%%%%%%%%%%%%%%%%%%%%%%%%%%%%%
with the second equality following from the fact that  $N$ is even. 

Combining \eqref{Eq:claim-beta-bb'} and \eqref{Eq:signat}, we obtain
\[
(-1)^{\sigma_{b\kappa_r}} \beta_{b\kappa_r}
=
(-1)^{b-r}(-1)^{r-1}\Delta_b
=
(-1)^{b-1}\Delta_b,
\]
which is independent of $r$. There are $d+1=N-2$ possible values of $r$, so \eqref{Eq:S1:2}  becomes
%%%%%%%%%%%%%%%%%%%%%%%%%%%%%%%%%%%%%%%%
\begin{equation}%%%%%%%%%%%%%%%%%%%%%%%%%%%%%%%%
%%%%%%%%%%%%%%%%%%%%%%%%%%%%%%%%%%%%%%%%
\label{Eq:S1-Delta-b}
S_1
=
2(N-2)d!
\sum_{b=2}^{N}
(-1)^{b-1}\Delta_b.
%%%%%%%%%%%%%%%%%%%%%%%%%%%%%%%%%%%%%%%%
\end{equation}%%%%%%%%%%%%%%%%%%%%%%%%%%%%%%%%
%%%%%%%%%%%%%%%%%%%%%%%%%%%%%%%%%%%%%%%%
Consequently, it remains only to prove that 
%%%%%%%%%%%%%%%%%%%%%%%%%%%%%%%%%%%%%%%%
\begin{equation}%%%%%%%%%%%%%%%%%%%%%%%%%%%%%%%%
%%%%%%%%%%%%%%%%%%%%%%%%%%%%%%%%%%%%%%%%
\label{Eq:Sum-Delta-b}
\sum_{b=2}^{N}
(-1)^{b-1}\Delta_b
=
0.
%%%%%%%%%%%%%%%%%%%%%%%%%%%%%%%%%%%%%%%%
\end{equation}%%%%%%%%%%%%%%%%%%%%%%%%%%%%%%%%
%%%%%%%%%%%%%%%%%%%%%%%%%%%%%%%%%%%%%%%%

By construction,
$
u_s-u_1
=
\dd\Log
\left( {[1\kappa_s][b\kappa_1]} / \big( 
     {[b\kappa_s][1\kappa_1]} \big)
\right)$, 
so every factor occurring in $\Delta_b$ is the logarithmic
differential of a cross-ratio. Hence each $\Delta_b$
is intrinsically defined on $\boldsymbol{\mathcal M}_{0,A}$.
Therefore \eqref{Eq:Sum-Delta-b} may be checked in any convenient
projective coordinate on $\mathbf P^1$. We choose a coordinate for which
$
p_1=\infty$, 
and write
\[
v_1=(1,0),
\qquad \mbox{ and } \quad 
v_j=(x_j,1),
\quad
j=2,\ldots,N
\]
for the coordinates of the lifts in $V$ of the points of the element of $\boldsymbol{\mathcal M}_{0,A}$ under consideration. For $i,j\geq2$, we thus have 
$ [1j]=1$ and 
$
[ij]=x_i-x_j
$ hence 
%We have 
%%%%%%%%%%%%%%%%%%%%%%%%%%%%%%%%%%%%%%%%%%%
\begin{equation}%%%%%%%%%%%%%%%%%%%%%%%%%%%%%%%%%%
%%%%%%%%%%%%%%%%%%%%%%%%%%%%%%%%%%%%%%%%%%%
\label{Eq:Db}
\Delta_b
=
\bigwedge_{s=2}^{d+1}
\dd\Log
\left(
\frac{x_b-x_{\kappa_1}}
     {x_b-x_{\kappa_s}}
\right)\, .
%%%%%%%%%%%%%%%%%%%%%%%%%%%%%%%%%%%%%%%%%%%
\end{equation}%%%%%%%%%%%%%%%%%%%%%%%%%%%%%%%%%%%
%%%%%%%%%%%%%%%%%%%%%%%%%%%%%%%%%%%%%%%%%%

In the affine coordinates $x_2,\ldots,x_N$, we consider the differential form $\Omega
=
\dd x_2\wedge\cdots\wedge\dd x_N$, the two vector fields $
T
=
\sum_{j=2}^{N}
\partial_{x_j}
%\frac{\partial}{\partial x_j}
$, 
$E
=
\sum_{j=2}^{N}
x_j\,\partial_{x_j} $, and we set 
\[
\Theta
=
\iota_E\iota_T\Omega.
\]

We set
$
K=\mathbf C(x_2,\ldots,x_N)
$
and denote by $W$ the subspace of  
$\Omega^1_K$ consisting of rational  $1$-forms  annihilated by both $T$ and $E$.
It is a $K$-vector space of dimension $(N-1)-2=d$. Since every factor in the
wedge product \eqref{Eq:Db} annihilates both $T$ and $E$, we have
\[
\Delta_b\in\wedge^d W .
\]

We claim that $\Theta\in\wedge^d W$ as well. To see this, let
$
(T,E,X_1,\ldots,X_d)
$
be a $K$-basis of the space $\Xi_K^1$ of rational vector fields in the
$x_i$'s, and let $
(\tau,\eta,\chi_1,\ldots,\chi_d)$
be the dual basis of $\Omega_K^1$. Then there exists $c\in K^*$ such that
$
\Omega
=
c\,\tau\wedge\eta\wedge\chi_1\wedge\cdots\wedge\chi_d
$, 
and therefore
$
\Theta
=
c\,\big( \wedge_{i=1}^d \chi_i\big)$. 
%\cdots\wedge\chi_d$. 
Since the $\chi_i$'s form a $K$-basis of $W$, we obtain
\[
\Theta\in\wedge^d W .
\]

Since $\dim_K W=d$, the space $\wedge^d W$ is one-dimensional. Hence
there exists a rational function $\varrho_b\in K$ such that
%Since $\dim_K W=d$, we have that $\wedge^{d} W$ is of dimension 1, hence there %exists  a rational function $\varrho_b$ such that
%%%%%%%%%%%%%%%%%%%%%%%%%%%%%%%%%%%%%%%%%%%
\begin{equation}%%%%%%%%%%%%%%%%%%%%%%%%%%%%%%%%%%
%%%%%%%%%%%%%%%%%%%%%%%%%%%%%%%%%%%%%%%%%%%
\label{Delta=F Thetha}
\Delta_b=\varrho_b\,\Theta\, .
%%%%%%%%%%%%%%%%%%%%%%%%%%%%%%%%%%%%%%%%%%%
\end{equation}%%%%%%%%%%%%%%%%%%%%%%%%%%%%%%%%%%%
%%%%%%%%%%%%%%%%%%%%%%%%%%%%%%%%%%%%%%%%%%%

To determine $\varrho_b$, we evaluate both sides of the above equality on the $d$-tuple of vector fields 
$\xi=\big(
\frac{\partial}{\partial x_{\kappa_s}}
\big)_{s=2}^{d+1}$. 
From \eqref{Eq:Db}, we obtain that $\Delta_b(\xi)$ is the determinant of the matrix whose $(s,t)$-th coefficient is equal to 
the $\dd\Log
\left(
\frac{x_b-x_{\kappa_1}}
     {x_b-x_{\kappa_s}}
\right)$ evaluated on $\partial/\partial {x_{\kappa_t}} $ (with $s,t=2,\ldots,d+1$). Hence we have 
%%%%%%%%%%%%%%%%%%%%%%%%%%%%%%%%%%%%%%%%%%%
%$i_{ \partial/\partial {x_{\kappa_t}}} \dd\Log
%\left(
%\frac{x_b-x_{\kappa_1}}
%     {x_b-x_{\kappa_s}}
%\right)$
%%%%%%%%%%%%%%%%%%%%%%%%%%%%%%%%%%%%%%%%%%%
\begin{equation}%%%%%%%%%%%%%%%%%%%%%%%%%%%%%%%%%%
%%%%%%%%%%%%%%%%%%%%%%%%%%%%%%%%%%%%%%%%%%%
\label{Eq:Delta-b-xi}
\Delta_b
\big(\,
\xi
\, \big)= \frac1{\prod_{s=2}^{d+1}
({x_b-x_{\kappa_s}})} \, .
%%%%%%%%%%%%%%%%%%%%%%%%%%%%%%%%%%%%%%%%%%%
\end{equation}%%%%%%%%%%%%%%%%%%%%%%%%%%%%%%%%%%%
%%%%%%%%%%%%%%%%%%%%%%%%%%%%%%%%%%%%%%%%%%%
%\[\Delta_b
%\big(\,
%\xi
%\, \big)=
%\prod_{s=2}^{d+1}
%\frac{1}{x_b-x_{\kappa_s}}\, .
%\]

On the other hand, we have 
$$ \Theta(\xi)= \Omega\Big( T,E, \partial_{x_{\kappa_2}},\ldots,\partial_{x_{\kappa_{d+1}}}\Big)\,. $$
%\[
%\Theta
%\left(
%\frac{\partial}{\partial x_{c_2}},
%\ldots,
%\frac{\partial}{\partial x_{c_m}}
%\right)
%=
%\Omega
%\left(
%T,E,
%\frac{\partial}{\partial x_{c_2}},
%\ldots,
%\frac{\partial}{\partial x_{c_m}}
%\right).
%\]
Only the components of $T$ and $E$ along
$\partial/\partial x_b$ and
$\partial/\partial x_{\kappa_1}$ contribute. The ordered list
$
(b,\kappa_1,\kappa_2,\ldots,\kappa_{d+1})
$
has signature $(-1)^{b-2}$ with respect to the increasing list
$
(2,3,\ldots,N),
$
while
\[
\det
\begin{pmatrix}
1 & x_b\\
1 & x_{\kappa_1}
\end{pmatrix}
=
x_{\kappa_1}-x_b\, .
\]
It follows that 
$
\Theta
(\xi)
=
(-1)^{b-1}(x_b-x_{\kappa_1})
$ and 
combining this with \eqref{Eq:Delta-b-xi}, we get 
%\[
%\varrho_b
%=
%\frac{(-1)^{b-1}}
%{\displaystyle
%\prod_{s=1}^{d+1}(x_b-x_{\kappa_s})}.
%\]
\[
\varrho_b
=
{(-1)^{b-1}} 
{\displaystyle
\prod_{s=1}^{d+1}(x_b-x_{\kappa_s})^{-1}}.
\]
Since $d+1=N-2$ is even, we thus have
%%%%%%%%%%%%%%%%%%%%%%%%%%%%%%%%%%%%%%%%%%%
\begin{equation}%%%%%%%%%%%%%%%%%%%%%%%%%%%%%%%%%%
\label{Eq:Delta-b-Theta}%%%%%%%%%%%%%%%%%%%%%%%%%%%%%%
%%%%%%%%%%%%%%%%%%%%%%%%%%%%%%%%%%%%%%%%%%%
\Delta_b
=
{(-1)^{b-1} } \, \Theta\,
{\displaystyle
\prod_{\substack{j=2\\j\neq b}}^{N}
(x_j-x_b)^{-1}} 
%%%%%%%%%%%%%%%%%%%%%%%%%%%%%%%%%%%%%%%%%%%
\end{equation}%%%%%%%%%%%%%%%%%%%%%%%%%%%%%%%%%%%
%%%%%%%%%%%%%%%%%%%%%%%%%%%%%%%%%%%%%%%%%%%
%
%%%%%%%%%%%%%%%%%%%%%%%%%%%%%%%%%%%%%%%%%%%%
%\begin{equation}%%%%%%%%%%%%%%%%%%%%%%%%%%%%%%%%%%
%\label{Eq:Delta-b-Theta}%%%%%%%%%%%%%%%%%%%%%%%%%%%%%%
%%%%%%%%%%%%%%%%%%%%%%%%%%%%%%%%%%%%%%%%%%%
%\Delta_b
%=
%\frac{(-1)^{b-1} }
%{\displaystyle
%\prod_{\substack{j=2\\j\neq b}}^{N}
%(x_j-x_b)} \, \Theta\,
%%%%%%%%%%%%%%%%%%%%%%%%%%%%%%%%%%%%%%%%%%%
%\end{equation}%%%%%%%%%%%%%%%%%%%%%%%%%%%%%%%%%%%
%%%%%%%%%%%%%%%%%%%%%%%%%%%%%%%%%%%%%%%%%%%
hence identity \eqref{Eq:Sum-Delta-b} is equivalent to 
%%%%%%%%%%%%%%%%%%%%%%%%%%%%%%%%%%%%%%%%%%%
\begin{equation}%%%%%%%%%%%%%%%%%%%%%%%%%%%%%%%%%%%
\label{Eq:foin}
%%%%%%%%%%%%%%%%%%%%%%%%%%%%%%%%%%%%%%%%%%%
\sum_{b=2}^{N}
%\frac{1}
{
\displaystyle
\prod_{\substack{j=2\\j\neq b}}^{N}
(x_j-x_b)^{-1}}=0\, . 
%%%%%%%%%%%%%%%%%%%%%%%%%%%%%%%%%%%%%%%%%%%
\end{equation}%%%%%%%%%%%%%%%%%%%%%%%%%%%%%%%%%%%
%%%%%%%%%%%%%%%%%%%%%%%%%%%%%%%%%%%%%%%%%%%

Setting 
$Q(t)
=
\prod_{j=2}^{N}(t-x_j)$, because   $N$ is even, we have 
$$
Q'(x_b)
=
\prod_{\substack{j=2\\j\neq b}}^{N}
(x_b-x_j)
=
\prod_{\substack{j=2\\j\neq b}}^{N}
(x_j-x_b)\, .
$$   
Hence \eqref{Eq:foin} 
follows from the Lagrange interpolation identity
$
\sum_{b=2}^{N}
{1}/{Q'(x_b)}
=
0
$.

This concludes the proof that  the form $\omega_A$, a priori defined on $\widetilde U_A$,  is invariant under the action of $\big( \mathbf C^*\big)^{N}$ 
and hence descends to the quotient  $\widetilde U_A\rightarrow U_A$. 
\medskip

It remains to prove the invariance of $\omega_A$ with respect to the action of ${\bf GL}(V)$. 
Since each factor is the logarithmic
differential of a cross-ratio, the forms
$\beta_\sigma$ are ${\bf GL}(V)$-invariant.  
On the other hand, for a given $g\in {\bf GL}(V)$, we have 
${\rm Log}\, \big\lvert\, [gv_i,gv_j]\, \big \rvert
={\rm Log}\, \lvert\, \det \, g\, \rvert 
+
{\rm Log}\, \big\lvert\, [ij]\, \big \rvert$ for every $i,j$, hence the variation of $\omega_A$ under $g$ is
%\[
%{\rm Log}\, \big\lvert\, [gv_i,gv_j]\, \big \rvert
%={\rm Log}\, \lvert\, \det \, g\, \rvert 
%+
%{\rm Log}\, \big\lvert\, [ij]\, \big \rvert
%\, .
%\]
%Hence the variation of $\omega_A$ under $g$ is
$$\frac{{\rm Log}|\det g \,|}{N!} 
\sum_{\sigma\in{\mathfrak S}_N}
(-1)^\sigma\beta_\sigma\, .
$$ 
But identity \eqref{Eq:MRP}, applied with
$
c_1=\cdots=c_N=1$, 
gives $
2
\sum_{\sigma\in{\mathfrak S}_N}
(-1)^\sigma \beta_\sigma
= 0$ which shows that $\omega_A$ is indeed  invariant under ${\bf GL}(V)$. Together with its
$\big(\mathbf C^*\big)^N$-invariance established before, this ensures that $\omega_A$ descends to $\boldsymbol{\mathcal M}_{0,A}$ and finishes the proof of the first point of Proposition 
\ref{Prop:descent}. 
\mk

%%%%%%%%%%%%%%%%%%%%%%%%%%%%%%%%%%%%%%%%%%
\subsection{\bf Proof of the second point of Proposition \ref{Prop:descent}}
%%%%%%%%%%%%%%%%%%%%%%%%%%%%%%%%%%%%%%%%%%
This will follow from the explicit expression of $\omega_A$ in affine coordinates
obtained in Section~\ref{S:Link-explicit} below; see
Corollary~\ref{Cor:Second-Point}.

%here-here-here

%%%%%%%%%%%%%%%%%%%%%%%%%%%%%%%%%%%%%%%%%%%%
%%%%%%%%%%%%%%%%%%%%%%%%%%%%%%%%%%%%%%%%%%%%
\section{\bf Proof of Theorem \ref{Prop:Differential-Identity}}
\label{S:Proof-of-Theorem}
%%%%%%%%%%%%%%%%%%%%%%%%%%%%%%%%%%%%%%%%%%%%
%%%%%%%%%%%%%%%%%%%%%%%%%%%%%%%%%%%%%%%%%%%%
We shall give two proofs of the differential identity
\eqref{Eq:sum-(-1)a-omega-a=0}. The first is conceptual: the vanishing
is obtained by antisymmetrizing a form which is invariant under an odd
permutation. The second is more explicit and is formulated at the
symbolic level. After choosing a convenient affine normalization, the
terms $f_a^*(\omega_a)$ are represented by elements of a formal symbol
space, in which the cancellation becomes transparent: for each
$b\in\{1,\ldots,N\}$, the same block $W_b$ appears in the symbol of
$f_b^*(\omega_b)$ and in the symbol of
$f_{N+1}^*(\omega_{N+1})$, with opposite coefficients in the
alternating sum.

%%%%%%%%%%%%%%%%%%%%%%%%%%%%%%%%%%%%%%%%%%%%%
%%%%%%%%%%%%%%%%%%%%%%%%%%%%%%%%%%%%%%%%%%%%%
\subsection{\bf First proof.}
We keep the notation introduced above. Thus $N\geq 4$ is even,
$
A=\{1,\ldots,N+1\}$, 
and, for every $a\in A$,
$
A_a=A\setminus\{a\}
$ 
is endowed with the ordering induced by the natural ordering of $A$.
We denote by
$
f_a:\M_{0,A}\longrightarrow \M_{0,A_a}
$
the forgetful map and put
$
\omega_a=\omega_{A_a}$. 

By Proposition \ref{Prop:descent}, the form $\omega_a$, initially
defined on $\widetilde U_a=\widetilde U_{A_a}$, descends to a well-defined nonzero
differential $(N-3)$-form on $\M_{0,A_a}$. In the present proof, however, it will be more convenient to work
entirely on the spaces of lifts. For every $a\in A$, let
\[
F_a:\widetilde U_A\longrightarrow\widetilde U_{a}
\]
be the map which forgets the vector $v_a$. If
\[
q_A:\widetilde U_A\longrightarrow\M_{0,A}
\qquad\text{and}\qquad
q_a=q_{A_a}:\widetilde U_{a}\longrightarrow\M_{0,A_a}
\]
denote the quotient maps, we have the commutative diagram
\[
\xymatrix{
\widetilde U_A \ar[r]^{F_a} \ar[d]_{q_A}
&
\widetilde U_{a} \ar[d]^{q_{a}}
\\
\M_{0,A} \ar[r]_{f_a}
&
\M_{0,A_a}.
}
\]
%ssss
To avoid any ambiguity between a differential form on a moduli space
and its pull-back to the corresponding space of lifts, we set
\[
\widetilde\omega_a
=
q_{a}^*(\omega_a)
%\qquad 
%\mbox{ thus } \qquad  F_a^*\big(\widetilde\omega_a\big)
%=
%q_A^*\bigl(f_a^*(\omega_a)\bigr)
\, .
\]
%In what follows, by a slight abuse of notation, we shall denote again
%by $\omega_a$ the pull-back $q_{a}^*(\omega_a)$ to
%$\widetilde U_{a}$. 
Thus
%\begin{equation}
%\label{Eq:Fa-fa}
$F_a^*\big(\widetilde\omega_a\big)
=
q_A^*\bigl(f_a^*(\omega_a)\bigr)$ for every $a$, and 
%\end{equation}
 since $q_A$ is a surjective submersion, pull-back by $q_A$ is
injective on differential forms. Consequently, in order to prove
\eqref{Eq:sum-(-1)a-omega-a=0}, it is enough to establish
\begin{equation}
\label{Eq:differential-identity-upstairs}
\sum_{a=1}^{N+1}
(-1)^aF_a^*\big(\widetilde\omega_a\big) =0
\end{equation}
on $\widetilde U_A$. We shall prove this identity by a direct antisymmetrization argument. The proof will follow from a simple auxiliary expression for the
forms $\widetilde\omega_a$.

\medskip
%%%%%%%%%%%%%%%%%%%%%%%%%%%%%%%%%%%%%%%%%%%%%
\noindent
{\bf 1. An auxiliary form.}
%%%%%%%%%%%%%%%%%%%%%%%%%%%%%%%%%%%%%%%%%%%%%
Let
$B=\{b_1<\cdots<b_N\}$
be any ordered set of cardinality $N$. We denote by
$
q_B:\widetilde U_B\longrightarrow\M_{0,B}
$
the quotient map and set
$
\widetilde\omega_B=q_B^*(\omega_B)$. 
By the original definition of $\omega_B$, one has
$
\widetilde\omega_B=\Alt_B(\eta_B),
$
where
\[
\eta_B
=
\Log\big\lvert[b_1b_2]\big\rvert
\bigwedge_{r=3}^{N-1}
\dd\Log\left(
\frac{[b_1b_r][b_2b_N]}
     {[b_2b_r][b_1b_N]}
\right).
\]

We introduce the auxiliary form
\begin{equation}
\label{Eq:gamma-B}
\gamma_B
=
\Log\big\lvert[b_1b_2]\big\rvert
\bigwedge_{r=3}^{N-1}
\dd\Log\left(
\frac{[b_1b_r]}{[b_2b_r]}
\right).
\end{equation}
The important feature of $\gamma_B$ is that it does not involve the
last label $b_N$, unlike $\eta_B$. We claim that
\begin{equation}
\label{Eq:omega-gamma}
\widetilde\omega_B
=
(N-2)\Alt_B(\gamma_B)\, .
\end{equation}
%which  is equivalent to $\eta_B-(N-2)\gamma_B \in {\rm Ker}\big( \Alt_B\big)$.

To prove this, for $r=3,\ldots,N$, put
%%%%%%%%%%%%%%%%%%%%%%%%%%%%%%%%%%%%%%%%
$$
\alpha_r^B = \dd\Log\big(
[b_1b_r]/[b_2b_r]\big)\, .
%\qquad  \mbox{ for }\,r=3,\ldots,N\,.
$$ 
%%%%%%%%%%%%%%%%%%%%%%%%%%%%%%%%%%%%%%%%
%\[
%\alpha_r^B
%=
%\dd\Log\left(
%\frac{[b_1b_r]}{[b_2b_r]}
%\right)\qquad \mbox{ for }\qquad 
% 3\leq r\leq N.
%\]
Then
%%%%%%%%%%%%%%%%%%%%%%%%%%%%%%%%%%%%%%%%
\begin{equation*}
%%%%%%%%%%%%%%%%%%%%%%%%%%%%%%%%%%%%%%%%
\gamma_B
= 
\Log\big\lvert[b_1b_2]\big\rvert\,
\wedge_{r=3}^{N-1} \alpha_r^B
%\alpha_3^B\wedge\cdots\wedge\alpha_{N-1}^B
\quad 
\mbox{ whereas }\quad 
\eta_B
=
\Log\big\lvert[b_1b_2]\big\rvert
\wedge_{r=3}^{N-1}
\left(\alpha_r^B-\alpha_N^B\right)\, .
%%%%%%%%%%%%%%%%%%%%%%%%%%%%%%%%%%%%%%%%
%%%%%%%%%%%%%%%%%%%%%%%%%%%%%%%%%%%%%%%%
\end{equation*}
%%%%%%%%%%%%%%%%%%%%%%%%%%%%%%%%%%%%%%%%

Expanding the latter wedge product gives
\begin{equation}
\label{Eq:expansion-alpha}
\eta_B
= \Log\big\lvert[b_1b_2]\big\rvert\,
\sum_{r=3}^{N}
(-1)^{N-r}
\alpha_3^B\wedge\cdots\wedge
\widehat{\alpha_r^B}\wedge\cdots\wedge\alpha_N^B.
\end{equation}

For $r=3,\ldots,N$, consider the cycle $
\tau_r=(b_r\,b_{r+1}\,\cdots\,b_N)\in{\mathfrak S}_B$. 
Its signature is $
(-1)^{\tau_r}=(-1)^{N-r}$, it 
 fixes $b_1$ and $b_2$, 
and we have $ \tau_r^*(\alpha_s^B)=\alpha_s^B$ for $s<r$ whereas 
 $ \tau_r^*(\alpha_t^B)=\alpha_{t+1}^B$ for $t=r,\ldots,N-1$. It follows that 
\[
\tau_r^*(\gamma_B)
=
\Log\big\lvert[b_1b_2]\big\rvert\,
\alpha_3^B\wedge\cdots\wedge
\widehat{\alpha_r^B}\wedge\cdots\wedge\alpha_N^B.
\]

It follows from \eqref{Eq:expansion-alpha} that 
\[
\eta_B-(N-2)\gamma_B
=
\sum_{r=3}^{N}
\left(
(-1)^{N-r}\tau_r^*(\gamma_B)-\gamma_B
\right).
\]
Since  $
\Alt_B(\nu^*\alpha)=(-1)^\nu\Alt_B(\alpha)$ 
for every $\nu\in\mathfrak S_B$, and since
$(-1)^{\tau_r}=(-1)^{N-r}$, each summand on the right-hand side
belongs to $\ker(\Alt_B)$. Therefore
\[
\Alt_B\Big(\eta_B-(N-2)\gamma_B\Big)=0,
\]
which is equivalent to \eqref{Eq:omega-gamma}. 
%\[
%\widetilde\omega_B
%=
%(N-2)\Alt_B(\gamma_B).
%\]
\medskip

%%%%%%%%%%%%%%%%%%%%%%%%%%%%%%%%%%%%%%%%%%%%%
\noindent
{\bf 2. Proof of the differential identity on $\boldsymbol{\widetilde U_A}$.}
%%%%%%%%%%%%%%%%%%%%%%%%%%%%%%%%%%%%%%%%%%%%%
We now return to $
A=\{1,\ldots,N+1\}$. Consider the form
%%%%%%%%%%%%%%%%%%%%%%%%%%%%%%%%%%%%%%%%%%%%%
\begin{equation}%%%%%%%%%%%%%%%%%%%%%%%%%%%%%%%%%%%%
\label{Eq:Gamma-A}%%%%%%%%%%%%%%%%%%%%%%%%%%%%%%%%%%
%%%%%%%%%%%%%%%%%%%%%%%%%%%%%%%%%%%%%%%%%%%%%
\Gamma
=
F_{N+1}^*(\gamma_{A_{N+1}})
=
\Log\big\lvert[12]\big\rvert
\bigwedge_{r=3}^{N-1}
\dd\Log\left(
\frac{[1r]}{[2r]}
\right)
%%%%%%%%%%%%%%%%%%%%%%%%%%%%%%%%%%%%%%%%%%%%%
\end{equation}%%%%%%%%%%%%%%%%%%%%%%%%%%%%%%%%%%%%%
%%%%%%%%%%%%%%%%%%%%%%%%%%%%%%%%%%%%%%%%%%%%%
on $\widetilde U_A$. The crucial point is that $\Gamma$ involves neither the label $N$
nor the label $N+1$. Hence, if  
$\tau\in{\mathfrak S}_A$ is 
the transposition exchanging $N$ and $N+1$,
we have
$
\tau^*\Gamma=\Gamma$.
 On the other hand,  we have 
$ \Alt_A(\tau^*\Gamma)
=
(-1)^\tau\Alt_A(\Gamma)
=
-\Alt_A(\Gamma)$. 
Therefore
\begin{equation}
\label{Eq:Alt-Gamma-zero}
\Alt_A(\Gamma)=0.
\end{equation}

We now decompose this vanishing according to the last value of the
permutations occurring in the antisymmetrization. By definition,
\begin{equation}
\label{Eq:Alt-Gamma}
\Alt_A(\Gamma)
=
\frac1{(N+1)!}
\sum_{\sigma\in{\mathfrak S}_A}
(-1)^\sigma\,\sigma^*\Gamma.
\end{equation}

Fix $a\in A$ and write  $
A_a=\{b_1<\cdots<b_N\}$.
 Let $\chi_a\in{\mathfrak S}_A$ be the permutation 
 moving $a$ to the last position while preserving the
relative order of the other labels: we have 
\[
\big(\chi_a(i)\big)_{i=1}^{N+1}
%=
%(b_1,\ldots,b_N,a)
=
\big(\, 1,\ldots,a-1, a+1,\ldots,N+1,a\, \big).
\]
 Hence  one has $(-1)^{\chi_a}=(-1)^{N+1-a}$. 
 %\begin{equation}
%\label{Eq:sign-chi-a}
%\sgn(\chi_a)
%=
%(-1)^{N+1-a}.
%\end{equation}
 Every permutation $\sigma\in{\mathfrak S}_A$ satisfying
$
\sigma(N+1)=a
$ 
is obtained uniquely by choosing a permutation
$\rho\in{\mathfrak S}_{N}\simeq {\mathfrak S}_{A_a}$ of the first $N$ entries. With respect to
the natural ordering of $A_a$, its one-line notation is
\[
\big(\sigma(i)\big)_{i=1}^{N+1} = 
\big(\, b_{\rho(1)}\, , \, \ldots\, , \, b_{\rho(N)}\, , \, a\, \big),
\]
and therefore $(-1)^\sigma =
(-1)^{N+1-a}(-1)^\rho$. 
%\begin{equation}
%\label{Eq:sign-sigma-rho}
%\sgn(\sigma)
%=
%(-1)^{N+1-a}\sgn(\rho).
%\end{equation}
Moreover, directly from the definitions of $\Gamma$ and
$\gamma_{A_a}$, we get 
\begin{equation}
\label{Eq:sigma-Gamma-Fa}
\sigma^*\Gamma
=
F_a^*\bigl(\rho^*\gamma_{A_a}\bigr).
\end{equation}

It follows that the contribution to \eqref{Eq:Alt-Gamma} of all
$\sigma\in{\mathfrak S}_{N+1}$ satisfying $\sigma(N+1)=a$ is
\[
\begin{aligned}
\frac1{(N+1)!}
\sum_{\substack{\sigma\in{\mathfrak S}_A\\
\sigma(N+1)=a}}
(-1)^\sigma\,\sigma^*\Gamma 
= & \, 
\frac{(-1)^{N+1-a}}{(N+1)!}
\sum_{\rho\in{\mathfrak S}_{A_a}}
(-1)^\rho\,
F_a^*(\rho^*\gamma_{A_a})
\\
= & \, 
\frac{(-1)^{N+1-a}}{N+1}
F_a^*\Big(\Alt_{A_a}(\gamma_{A_a})\Big).
\end{aligned}
\]
Summing over $a$, we obtain
\begin{equation}
\label{Eq:Alt-Gamma-faces}
\Alt_A(\Gamma)
=
\frac1{N+1}
\sum_{a=1}^{N+1}
(-1)^{N+1-a}
F_a^*\Big(\Alt_{A_a}\big(\gamma_{A_a}\big)\Big).
\end{equation}

%Since $N$ is even, $
%(-1)^{N+1-a}=(-1)^{a+1}$. 
Combining \eqref{Eq:Alt-Gamma-zero} and
\eqref{Eq:Alt-Gamma-faces} gives
\[
\sum_{a=1}^{N+1}
(-1)^a
F_a^*\Big(\Alt_{A_a}(\gamma_{A_a})\Big)
=
0.
\]
By \eqref{Eq:omega-gamma},
$
\widetilde\omega_a
=
(N-2)\Alt_{A_a}(\gamma_{A_a})$, 
and hence we obtain the desired  identity: 
\begin{equation}
\label{Eq:differential-identity-upstairs-final}
\sum_{a=1}^{N+1}
(-1)^aF_a^*\big(\widetilde\omega_a\big)
=
0\, .
\end{equation}
%which implies the wanted identity
%\begin{equation}
%\label{Eq:differential-identity-upstairs-final}
%\sum_{a=1}^{N+1}
%(-1)^af_a^*\big(\omega_a\big)
%=
%0\, . 
%\end{equation}

This proves the first assertion of Theorem 
\ref{Prop:Differential-Identity}. Since every $\omega_a$ is nonzero
by Proposition \ref{Prop:descent}, and each $f_a$ is a submersion,
the forms $f_a^*(\omega_a)$ are nonzero. 
Since we have $\sum_{a=1}^{N+1}
(-1)^af_a^*\big(\omega_a\big)
=
0$, it follows that the $A$-tuple 
\[
\boldsymbol{E}_{N-2}
=
\Bigl(
(-1)^af_a^*(\omega_a)
\Bigr)_{a\in A}\,
\]
can be seen as a global complete abelian relation for $\boldsymbol{\mathcal W}_{0,A}$.
\mk 

%%%%%%%%%%%%%%%%%%%%%%%%%%%%%%%%%%%%%%%%%%%%%
\noindent
{\bf 3. Equivariance under $\boldsymbol{\boldsymbol{\mathfrak S}_A}$.}
%%%%%%%%%%%%%%%%%%%%%%%%%%%%%%%%%%%%%%%%%%%%%
 Let $\sigma\in{\mathfrak S}_A$ and $a\in A$. Restriction of $\sigma$
gives a bijection between the ordered sets
$
\overline\sigma_a:
A_a\longrightarrow A_{\sigma(a)}$,
$
i\longmapsto\sigma(i)$ whose signature 
is
\begin{equation}
\label{Eq:signature-restriction}
(-1)^{\overline\sigma_a}
=
(-1)^{a+\sigma(a)}(-1)^\sigma\, .
\end{equation}
Indeed, consider the permutation matrix associated with $\sigma$.
Deleting the row indexed by $a$ and the column indexed by
$\sigma(a)$ gives the permutation matrix associated with
$\overline\sigma_a$. Expanding the determinant along the row indexed
by $a$ gives
\[
(-1)^\sigma
=
(-1)^{a+\sigma(a)}
(-1)^{\overline\sigma_a},
\]
which is equivalent to \eqref{Eq:signature-restriction}.
 The bijection $\overline\sigma_a$ induces a relabelling map
$
R_{\sigma,a}:
\widetilde U_{A_{\sigma(a)}}
\longrightarrow
\widetilde U_{A_a}$ and 
the forgetful maps are compatible with these relabellings: 
the following diagram 
\[
\xymatrix{
\widetilde U_A
\ar[r]^{\sigma}
\ar[d]_{F_{\sigma(a)}}
&
\widetilde U_A
\ar[d]^{F_a}
\\
\widetilde U_{{\sigma(a)}}
\ar[r]_{R_{\sigma,a}}
&
\widetilde U_{a}
}
\]
is commutative. 
Consequently,
\begin{equation}
\label{Eq:permutation-forgetful}
\sigma^*\circ F_a^*
=
F_{\sigma(a)}^*\circ R_{\sigma,a}^*.
\end{equation}

By the very definition of $\widetilde\omega_B$ as an
antisymmetrization, relabelling gives
%%%%%%%%%%%%%%%%%%%%%%%%%%%%%%%
\begin{equation}%%%%%%%%%%%%%%%%%%%%%%
\label{Eq:omega-relabel}%%%%%%%%%%%%%%%%%%
%%%%%%%%%%%%%%%%%%%%%%%%%%%%%%%
R_{\sigma,a}^*\big(\widetilde\omega_a\big)
=
(-1)^{\overline\sigma_a}\,
\widetilde\omega_{\sigma(a)}.
%%%%%%%%%%%%%%%%%%%%%%%%%%%%%%%
\end{equation}%%%%%%%%%%%%%%%%%%%%%%%
%%%%%%%%%%%%%%%%%%%%%%%%%%%%%%%
Combining \eqref{Eq:permutation-forgetful},
\eqref{Eq:omega-relabel}, and
\eqref{Eq:signature-restriction}, we obtain
%%%%%%%%%%%%%%%%%%%%%%%%%%%%%%%
\[
\begin{aligned}
\sigma^*
\Big(F_a^*(\widetilde\omega_a)\Big)
&=
F_{\sigma(a)}^*
\Big(
R_{\sigma,a}^*(\widetilde\omega_a)
\Big)
\\
&=
(-1)^{\overline\sigma_a}\,
F_{\sigma(a)}^*
\big(\,\widetilde\omega_{\sigma(a)}\,\big)
\\
&=
(-1)^{a+\sigma(a)}
(-1)^\sigma\,
F_{\sigma(a)}^*
\big(\,\widetilde\omega_{\sigma(a)}\,\big).
\end{aligned}
\]
%%%%%%%%%%%%%%%%%%%%%%%%%%%%%%%
Since all the forms involved descend, the corresponding identity on
$\M_{0,A}$ is 
$\sigma^*
\Big(f_a^*(\omega_a)\Big)
=
(-1)^{a+\sigma(a)}
(-1)^\sigma\,
f_{\sigma(a)}^*\big(\omega_{\sigma(a)}\big)$. 
%\begin{equation*}
%\label{Eq:component-equivariance}
%\sigma^*
%\Big(f_a^*(\omega_a)\Big)
%=
%(-1)^{a+\sigma(a)}
%(-1)^\sigma\,
%f_{\sigma(a)}^*\big(\omega_{\sigma(a)}\big)\, .
%\end{equation*}
 Multiplying by $(-1)^a$ gives
\[
\begin{aligned}
\sigma^*\Bigl(
(-1)^af_a^*(\omega_a)
\Bigr)
%&=
%(-1)^a
%(-1)^{a+\sigma(a)}
%(-1)^\sigma\,
%f_{\sigma(a)}^*(\omega_{\sigma(a)})
%\\
&=
(-1)^\sigma\,
(-1)^{\sigma(a)}
f_{\sigma(a)}^*\big(\omega_{\sigma(a)}\big).
\end{aligned}
\]
Replacing $a$ by $\sigma^{-1}(a)$, this reads
$\big(\sigma\cdot\boldsymbol{E}_{N-2}\big)_a=\sigma^*\big((\boldsymbol{E}_{N-2})_{\sigma^{-1}(a)}\big)
=(-1)^\sigma(\boldsymbol{E}_{N-2})_a$ for every $a\in A$, that is
$
\sigma\cdot\boldsymbol{E}_{N-2}
=
(-1)^\sigma \boldsymbol{E}_{N-2}$. 
This proves the second assertion and completes the proof of
Theorem  \ref{Prop:Differential-Identity}.

%\newpage
%.
%\newpage
%%%%%%%%%%%%%%%%%%%%%%%%%%%%%%%%%%%%%%%%%%%
%%%%%%%%%%%%%%%%%%%%%%%%%%%%%%%%%%%%%%%%%%%
\subsection{\bf A second proof: the symbolic approach}
%%%%%%%%%%%%%%%%%%%%%%%%%%%%%%%%%%%%%%%%%%%
%%%%%%%%%%%%%%%%%%%%%%%%%%%%%%%%%%%%%%%%%%%
\label{SS:Second-proof}
We give a second, more explicit proof of the differential identity
\eqref{Eq:sum-(-1)a-omega-a=0}. Whereas the first proof explains the
vanishing by means of a global antisymmetrization argument, the proof
below makes the cancellation explicit after choosing a convenient
affine normalization. The logarithmic coefficients and their
logarithmic differentials will be encoded by formal generators, and
the identity will first be established in the resulting symbol space.

We keep all the notation introduced in the first proof. In particular,
\[
A=\{1,\ldots,N+1\},
\qquad
A_a=A\setminus\{a\},
\]
and, for every ordered set $B$ of cardinality $N$, the form
$\gamma_B$ is defined by \eqref{Eq:gamma-B}. We shall use throughout
the identity
\begin{equation}
\label{Eq:omega-gamma-recalled}
\widetilde\omega_B
=
(N-2)\Alt_B(\gamma_B),
\end{equation}
proved in \eqref{Eq:omega-gamma}.

\begin{proof}

%%%%%%%%%%%%%%%%%%%%%%%%%%%%%%%%%%%%%%%%%%%
\medskip
\noindent
\textbf{Step 1. Affine normalization and the symbol space.}
%%%%%%%%%%%%%%%%%%%%%%%%%%%%%%%%%%%%%%%%%%%
%{\bf NEW}
We use the standard affine coordinates on $\M_{0,A}$ obtained by
normalizing
\[
p_{N-1}=0,
\qquad
p_N=1,
\qquad
p_{N+1}=\infty.
\]
We write
$
p_i=z_i$,
for $1\leq i\leq N-2$, 
and, for notational uniformity, we put
$z_{N-1}=0$ and $
z_N=1$. Thus we get a standard identification
\[
\M_{0,A}
\simeq
\left\{\, 
\big(z_i\big)_{i=1}^{N-2}\in\C^{N-2}
\,  \big\lvert \, \, 
z_i\neq0,1,\quad z_i\neq z_j\, 
\right\}.
\]

We lift this normalized projective configuration by choosing 
\[
v_i=(z_i,1)
\quad  \mbox{for } 1\leq i\leq N
\qquad\mbox{ and }\qquad
v_{N+1}=(1,0).
\]
Hence  the associated determinants satisfy $[ij]=z_i-z_j$ for  $1\leq i,j\leq N$, whereas $[N+1,i]=-[i,N+1]=1$ for every $i\leq N$. 
In particular, we have 
$\Log |\, [i,N+1]\, |=0$  and $\dd\Log\big({[i,N+1]}/({[j,N+1]})\big)=0$. 
%\[
%\Log \big|\, [i,N+1]\, \big|=0
%\qquad \mbox{ and } 
%\qquad 
%\dd\Log\left(\frac{[i,N+1]}{[j,N+1]}\right)=0.
%\]
\sk

We define
$$
\mathfrak E=
\Big\{ \, 
\{i,j\}\subset\{1,\ldots,N\}
\, \, \big\lvert  \,  \, 
i<j,\, 
\{i,j\}\neq\{N-1,N\}\, 
\Big\}
$$
and 
for $e=\{i,j\}\in\mathfrak E$, we set 
$$
\theta_e
=
\dd\Log(z_i-z_j)\,,$$ 
keeping in mind that  replacing $z_i-z_j$ by its negative
does not change its logarithmic differential  (since $
\dd\Log(-f)=\dd(-f)/(-f)=\dd f/f=\dd\Log(f)$ for any $f$). \sk

Given $e=\{i,j\}\in\mathfrak E$, let $D_e$ be the boundary divisor of the 
Deligne--Mumford compactification $\overline{\M}_{0,A}$ corresponding to the partition 
$ \{i,j\}\sqcup A \setminus\{i,j\}$ of $A$. 
%\[
%\{i,j\}\mid A\setminus\{i,j\}.
%\]
At the generic point of $D_e$, the function $z_i-z_j$ is a local
equation for $D_e$. Consequently, for any $e,e'\in \mathfrak E$, 
$\Res_{D_e}(\theta_{e'})=1$ if $e=e'$ and is zero otherwise. 
%%%%%%%%%%%%%%%%%%%%%%%%%%%%%%%%%%%
%\begin{equation}
%\label{Eq:residue-theta}
%%%%%%%%%%%%%%%%%%%%%%%%%%%%%%%%%%%
%\Res_{D_e}(\theta_{e'})
%=
%\begin{cases}
%\, 1& \mbox{if } e=e',\\
%\, 0& \mbox{otherwise}.
%\end{cases}
%%%%%%%%%%%%%%%%%%%%%%%%%%%%%%%%%%%
%\end{equation}%%%%%%%%%%%%%%%%%%%%%%%%%%%
%%%%%%%%%%%%%%%%%%%%%%%%%%%%%%%%%%%
One deduces from this that the rational 1-forms $\theta_e$, $e\in\mathfrak E$, are linearly
independent over $\mathbf C$. 
\sk 

We introduce the vector space
\[
\boldsymbol{\mathcal H}
=
\bigoplus_{e\in\mathfrak E}\C\,\varepsilon_e
\]
and identify the span of the $\theta_e$'s with $\boldsymbol{\mathcal H}$ by means of the
selected residue coordinates
\begin{equation}
\label{Eq:selected-residue-map}
\theta_e\longmapsto\varepsilon_e.
\end{equation}
All logarithmic differentials occurring below belong to this span.
\sk 

We shall also keep track of the logarithmic coefficients. For this purpose, we introduce a
second formal vector space
\[
\boldsymbol{\mathcal L}
=
\bigoplus_{e\in\mathfrak E}\C\,{\ell}_e,
\]
where $\ell_{\{i,j\}}$ records the function
$L_{ij}
=
\Log|z_i-z_j|$. 
Since $z_{N-1}-z_N=-1$, 
we adopt the conventions
$\ell_{\{N-1,N\}}=0$
and 
$\varepsilon_{\{N-1,N\}}=0$, in $\boldsymbol{\mathcal L}$ and $\boldsymbol{\mathcal H}$ respectively. Note that these two spaces are isomorphic, via the linear map $\boldsymbol{\mathcal L}\rightarrow \boldsymbol{\mathcal H}$, ${\ell}_e\mapsto \varepsilon_e$. 

Put
\begin{equation}
\label{Eq:decorated-residue-space}
\boldsymbol{\mathcal S}
=
\boldsymbol{\mathcal L}\otimes\wedge^{N-3}\boldsymbol{\mathcal H}\, .
\end{equation}

For $k\geq 0$, we denote by $\Omega^{k}(\M_{0,A})$ the space of holomorphic differential $k$-forms on 
$\M_{0,A}$, and by  $\Omega_{\mathrm{an}}^{k}(\M_{0,A})$ the space of $k$-forms with real-analytic coefficients,  that is, forms that can be written as linear combinations of forms of the type
$\varphi\,\eta$ with $\eta\in \Omega^{k}(\M_{0,A})$ and $\varphi\in C^\omega(\M_{0,A},\mathbf C)$. 
There is a natural linear realisation map
\begin{equation}
\label{Eq:evaluation-decorated}
%\rho:
\boldsymbol{\mathfrak s} : 
\boldsymbol{\mathcal S}
\longrightarrow
\Omega_{\mathrm{an}}^{N-3}(\M_{0,A})
\end{equation}
induced by the two linear maps  
$$\boldsymbol{\mathcal L} \rightarrow C^\omega(\M_{0,A},\mathbf C)
\qquad \mbox{ and } \qquad 
\boldsymbol{\mathcal H}  \rightarrow \Omega^1(\M_{0,A})$$ 
defined respectively by
$\ell_{\{i,j\}}
\longmapsto
\Log \big|\, z_i-z_j\, \big|
$
and  $
\varepsilon_{\{i,j\}}
\longmapsto
\dd\Log \big(z_i-z_j\big)$. 
%\[
%\ell_{\{i,j\}}
%\longmapsto
%\Log \big|\, z_i-z_j\, \big|
%\qquad
%\mbox{and } \qquad 
%\varepsilon_{\{i,j\}}
%\longmapsto
%\dd\Log \big(z_i-z_j\big)\, .
%\]
   We shall prove the differential identity already in the formal space $\boldsymbol{\mathcal S}$. Applying $\boldsymbol{\mathfrak s}$ will then give the desired
identity of differential forms.
\sk

An element
$\hat\alpha\in\boldsymbol{\mathcal S}$ will be called a \emph{symbol} of a
differential form $\alpha$ if
\[
\boldsymbol{\mathfrak s}\big(\hat\alpha\big)=\alpha.
\]
No uniqueness of such a symbol, and in particular no injectivity of
$\boldsymbol{\mathfrak s}$, will be needed below. It will be enough to construct specific
symbols for the forms $f_a^*(\omega_a)$ and to prove that their
alternating sum vanishes already in $\boldsymbol{\mathcal S}$.
\medskip

%%%%%%%%%%%%%%%%%%%%%%%%%%%%%%%%%%%%%%%%%%%
\noindent
\textbf{Step 2. The symbol associated with an ordering.}
%%%%%%%%%%%%%%%%%%%%%%%%%%%%%%%%%%%%%%%%%%%
For $b\in\{1,\ldots,N\}$, set
$
B_b
=
\{1,\ldots,N\}\setminus\{b\}$ and let 
$\mu=(i_1,\ldots,i_{N-1})
$
be an ordering of $B_b$. We associate with $\mu$ the element
\begin{equation}
\label{Eq:Gamma-hat-mu}
\widehat\Gamma_\mu
=
\ell_{\{i_1,i_2\}}
\otimes
\bigwedge_{r=3}^{N-1}
\left(
\varepsilon_{\{i_1,i_r\}}
-
\varepsilon_{\{i_2,i_r\}}
\right)
\in\boldsymbol{\mathcal S}\, .
\end{equation}
By construction, we have 
$
\boldsymbol{\mathfrak s}(\widehat\Gamma_\mu)
=
\Gamma_\mu,
$ with 
\begin{equation}
\label{Eq:Gamma-mu}
\Gamma_\mu
=
{\rm Log}\, \big|\, z_{i_1}-z_{i_2}\,\big|
\bigwedge_{r=3}^{N-1}
\dd\Log\left(
\frac{z_{i_1}-z_{i_r}}
     {z_{i_2}-z_{i_r}}
\right)
\in \Omega_{\rm an}^{N-3}\big(\M_{0,A}\big)\, .
\end{equation}

Let $(-1)^\mu$ be the signature of the permutation of the naturally ordered set $B_b$
whose one-line notation is
$
(i_1,\ldots,i_{N-1})$. 
Finally, set
\begin{equation}
\label{Eq:W-b}
W_b
=
\sum_{\mu\in\operatorname{Ord}(B_b)}
(-1)^\mu\,
\widehat\Gamma_\mu
\in\boldsymbol{\mathcal S}\,
\end{equation}
where the sum is over the set $\operatorname{Ord}(B_b)$ of all orderings of 
$B_b$. (Note that we have a bijection $\operatorname{Ord}(B_b)\simeq 
{\mathfrak S}_{B_b}$). 
 Thus $W_b$ is the alternating symbolic block associated with the
ordered set $B_b$. We shall see that exactly the same block occurs in
the symbols of $f_b^*(\omega_b)$ and
$f_{N+1}^*(\omega_{N+1})$. 
\medskip
\medskip
\noindent
\textbf{Step 3. The symbols of
$\boldsymbol{f_b^*(\omega_b)}$, $\boldsymbol{1\leq b\leq N}$.}
%%%%%%%%%%%%%%%%%%%%%%%%%%%%%%%%%%%%%%%%%%%
Fix $b\in\{1,\ldots,N\}$. Recall that
$
A_b=A\setminus\{b\}
$
is endowed with the ordering induced by the natural ordering of $A$.
Since $b\leq N$, the label $N+1$ belongs to $A_b$.

Write $
A_b=\{c_1<\cdots<c_N\}$. 
For $\sigma\in\mathfrak S_{A_b}$, we shall 
write its  one-line notation as 
\[
(\sigma_1,\ldots,\sigma_N)
=
\bigl(\sigma(c_1),\ldots,\sigma(c_N)\bigr).
\]
By \eqref{Eq:omega-gamma-recalled}, we have
$\widetilde\omega_b
=
({N-2}/{N!})
\sum_{\sigma\in\mathfrak S_{A_b}}
(-1)^\sigma\,
\sigma^*(\gamma_{A_b})$.
%\[
%\widetilde\omega_b
%=
%\frac{N-2}{N!}
%\sum_{\sigma\in\mathfrak S_{A_b}}
%(-1)^\sigma\,
%\sigma^*(\gamma_{A_b}).
%\]
After pulling back by $F_b$, this gives
\begin{equation}
\label{Eq:omega-b-expanded-symbolic}
F_b^*\big(\,\widetilde\omega_b\,\big)
=
\frac{N-2}{N!}
\sum_{\sigma\in\mathfrak S_{A_b}}
(-1)^\sigma\,
F_b^*\Big(\sigma^*\big(\gamma_{A_b}\big)\Big)\, .
\end{equation}

Let us now write explicitly the terms occurring in this sum.
By the definition \eqref{Eq:gamma-B} of $\gamma_B$, we have
\begin{equation}
\label{Eq:sigma-gamma-Ab-explicit}
F_b^*\Big(\sigma^*(\gamma_{A_b})\Big)
=
\Log\,\big|\,[\sigma_1\sigma_2]\,\big|
\bigwedge_{r=3}^{N-1}
\dd\Log\left(
\frac{[\sigma_1\sigma_r]}
     {[\sigma_2\sigma_r]}
\right),
\end{equation}
where the brackets on the right-hand side are now viewed as functions
on $\widetilde U_A$.\sk 

We restrict \eqref{Eq:sigma-gamma-Ab-explicit} to the affine section
chosen in Step~1, namely 
$v_i=(z_i,1)$ for $1\leq i\leq N$,  and
$
v_{N+1}=(1,0)$.
There are now three possibilities, according to the position occupied
by the label $N+1$ in the one-line notation
$(\sigma_1,\ldots,\sigma_N)$.
\sk

\begin{itemize} 
\item
If $N+1\in\{\sigma_1,\sigma_2\}$, then the logarithmic coefficient in
\eqref{Eq:sigma-gamma-Ab-explicit} is
$
\Log\bigl|[N+1,j]\bigr|
=
\Log 1
=
0$ for some $j\in\{1,\ldots,N\}$, hence $F_b^*\big(\sigma^*(\gamma_{A_b})\big)=0$.
\sk
%%%%%%%%%%
\item  If $\sigma_r=N+1$ for some
$r\in\{3,\ldots,N-1\}$,  both $[\sigma_1,N+1]$ and $[\sigma_2,N+1]$ are equal to $-1$ hence 
$ 
\dd\Log\big(
{[\sigma_1,N+1]} / 
     {[\sigma_2,N+1]}
\big)
=
0$.  Thus we have $
F_b^*\bigl(\sigma^*(\gamma_{A_b})\bigr)=0
$ again.\sk
%%%%%%%%%%%%%%%%%%%%%%%%%%%%%%
%\item 
\end{itemize}
\smallskip

It remains to consider the case when $\sigma_N=N+1$.  In this case
$$\mu = (i_1,\ldots,i_{N-1})=
(\sigma_1,\ldots,\sigma_{N-1})
$$ 
is an ordering of $B_b=\{1,\ldots,N\}\setminus\{b\}= A\setminus \{ \,b\, , \, N+1\,\}$, all the labels $\sigma_1,\ldots,\sigma_{N-1}$ belong to
$\{1,\ldots,N\}$, and in the affine coordinates we are working with, 
\eqref{Eq:sigma-gamma-Ab-explicit}  is written 
\begin{align}
F_b^*\bigl(\sigma^*(\gamma_{A_b})\bigr)
&=
\Log \big|\, z_{i_1}-z_{i_2}\, \big|
\bigwedge_{r=3}^{N-1}
\dd\Log\left(
\frac{z_{i_1}-z_{i_r}}
     {z_{i_2}-z_{i_r}}
\right)
%\nonumber
 = 
\Gamma_\mu\, .
\label{Eq:sigma-gamma-Ab-Gamma-mu}
\end{align}
Conversely, every ordering
$\mu=(i_1,\ldots,i_{N-1})$ of $B_b$ arises in this way from the unique
permutation of $A_b$ whose one-line notation is
$
\big(\sigma_1,\ldots,\sigma_{N-1},N+1\big)$.
 Since $N+1$ is the largest element of the naturally ordered set
$A_b$, appending it in the last position does not create any
additional inversion. Therefore the signature of this permutation is
exactly $(-1)^\mu$.
\sk

It follows from \eqref{Eq:omega-b-expanded-symbolic} that, on the
section corresponding to the affine coordinates $z_i$, we have 
$
F_b^*\big(\,\widetilde\omega_b\,\big)
= ({N-2}/{N!}) 
\sum_{\mu\in\operatorname{Ord}(B_b)}
(-1)^\mu\,\Gamma_\mu
$. 
%\begin{equation}
%\label{Eq:omega-b-affine-Gamma}
%F_b^*\big(\,\widetilde\omega_b\,\big)
%=
%\frac{N-2}{N!}
%\sum_{\mu\in\operatorname{Ord}(B_b)}
%(-1)^\mu\,\Gamma_\mu.
%\end{equation}
 Replacing each $\Gamma_\mu$ by its symbol
$\widehat\Gamma_\mu$, defined in \eqref{Eq:Gamma-hat-mu}, we obtain
the following symbol of $f_b^*(\omega_b)$:
%%%%%%%%%%%%%%%%%%%%%%%%%%%%%%%%%%%%%%%%%%%%%
\begin{equation}%%%%%%%%%%%%%%%%%%%%%%%%%%%%%%%%%%%%
\label{Eq:symbol-omega-b}%%%%%%%%%%%%%%%%%%%%%%%%%%%%%%
%%%%%%%%%%%%%%%%%%%%%%%%%%%%%%%%%%%%%%%%%%%%%
\widehat\Omega_b
=
\frac{N-2}{N!} \hspace{-0.2cm} 
\sum_{\mu\in\operatorname{Ord}(B_b)}
(-1)^\mu\,\widehat\Gamma_\mu
=
\frac{N-2}{N!}\,W_b \, .
\end{equation}
Equivalently, for $b=1,\ldots,N$, we have 
\begin{equation}
\label{Eq:realization-omega-b}
\boldsymbol{\mathfrak s}\big(\,\widehat\Omega_b\, \big)
=
f_b^*(\omega_b)\, .
%%%%%%%%%%%%%%%%%%%%%%%%%%%%%%%%%%%%%%%%%%%%%
\end{equation}%%%%%%%%%%%%%%%%%%%%%%%%%%%%%%%%%%%%%
%%%%%%%%%%%%%%%%%%%%%%%%%%%%%%%%%%%%%%%%%%%%%

%%%%%%%%%%%%%%%%%%%%%%%%%%%%%%%%%%%%%%%%%%%
\medskip
\noindent
\textbf{Step 4. The symbol of
$\boldsymbol{f_{N+1}^*(\omega_{N+1})}$.}
%%%%%%%%%%%%%%%%%%%%%%%%%%%%%%%%%%%%%%%%%%%
We now determine a  symbol of the last term
$
f_{N+1}^*(\omega_{N+1})
$
of the differential identity.  Recall that
$
A_{N+1}=\{1,\ldots,N\}
$.  
\sk

By \eqref{Eq:omega-gamma-recalled}, we have
$
\widetilde\omega_{N+1}
=
({N-2}/{N!})
\sum_{\sigma\in{\mathfrak S}_N}
(-1)^\sigma\,
\sigma^*\big(\gamma_{A_{N+1}}\big)
$.
%\[
%\widetilde\omega_{N+1}
%=
%\frac{N-2}{N!}
%\sum_{\sigma\in{\mathfrak S}_N}
%(-1)^\sigma\,
%\sigma^*\big(\gamma_{A_{N+1}}\big)\, .
%\]
 We group the terms in this antisymmetrization according to the last
value of $\sigma$, that we denote by 
$
b=b_\sigma=\sigma(N)\in\{1,\ldots,N\}$. 
Once $b$ is fixed, the first $N-1$ values of $\sigma$ form an ordering 
$\mu=(i_1,\ldots,i_{N-1})$
of 
$B_b=\{1,\ldots,N\}\setminus\{b\}$.

We now compare the signature of $\sigma$ with the signature
$(-1)^\mu$ introduced in Step~2. Let $\chi_b\in{\mathfrak S}_N $ such that 
\[
\big(\chi_b(i)\big)_{i=1}^{N}
=
(1,\ldots,b-1,b+1,\ldots,N,b)\, . 
\]
 The permutation $\chi_b$ moves $b$ to the last
position while preserving the relative order of the other labels.
Since this requires $N-b$ adjacent transpositions, we have
\begin{equation}
\label{Eq:sign-chi-b-second-proof}
(-1)^{\chi_b}
=
(-1)^{N-b}.
\end{equation}
Consequently, the permutation of $\{1,\ldots,N\}$ whose one-line notation is
\[
(i_1,\ldots,i_{N-1},b)
\]
has signature $
(-1)^{N-b}(-1)^\mu$.  The definition \eqref{Eq:gamma-B} has the important property that
$\gamma_B$ does not involve the last label of the ordered set $B$.
Hence, for the permutation $\sigma$ such that 
\[
\bigl(\sigma(1),\ldots,\sigma(N)\bigr)
=(i_1,\ldots,i_{N-1},b),
%\quad  \textcolor{red}{\bf ???}
\]
the form $\sigma^*(\gamma_{A_{N+1}})$ depends only on the ordered
$(N-1)$-tuple
\[
\mu=(i_1,\ldots,i_{N-1}),
\]
and, after restriction to the affine section fixed in Step~1, it is
precisely the form $\Gamma_\mu$ defined in
\eqref{Eq:Gamma-mu}.
 It follows that an element of $\boldsymbol{\mathcal S}$
whose realisation is 
$f_{N+1}^*(\omega_{N+1})$ is  given by 
%%%%%%%%%%%%%%%%%%%%%%%%%%%%%%%%%%%%%%%%
\begin{align}%%%%%%%%%%%%%%%%%%%%%%%%%%%%%%%%%
%%%%%%%%%%%%%%%%%%%%%%%%%%%%%%%%%%%%%%%%
\label{Eq:symbol-omega-N+1}
%%%%%%%%%%%%%%%%%%%%%%%%%%%%%%%%%%%%%%%%
\widehat\Omega_{N+1}
%&
=
\frac{N-2}{N!}
\sum_{b=1}^{N}
\sum_{\mu\in\operatorname{Ord}(B_b)}
(-1)^{N-b}(-1)^\mu\,
\widehat\Gamma_\mu
%\nonumber\\
%&
=
\frac{N-2}{N!}
\sum_{b=1}^{N}
(-1)^{N-b}W_b.
%%%%%%%%%%%%%%%%%%%%%%%%%%%%%%%%%%%%%%%%
\end{align}
%%%%%%%%%%%%%%%%%%%%%%%%%%%%%%%%%%%%%%%%
Equivalently,  we have 
%%%%%%%%%%%%%%%%%%%%%%%%%%%%%%%%%%%%%%%%
\begin{equation}%%%%%%%%%%%%%%%%%%%%%%%%%%%%%%%
\label{Eq:realization-omega-N+1}
%%%%%%%%%%%%%%%%%%%%%%%%%%%%%%%%%%%%%%%%
\boldsymbol{\mathfrak s}\big (\widehat\Omega_{N+1} \big)
=
f_{N+1}^*(\omega_{N+1}).
%%%%%%%%%%%%%%%%%%%%%%%%%%%%%%%%%%%%%%%%
\end{equation}%%%%%%%%%%%%%%%%%%%%%%%%%%%%%%%%
%%%%%%%%%%%%%%%%%%%%%%%%%%%%%%%%%%%%%%%%

%%%%%%%%%%%%%%%%%%%%%%%%%%%%%%%%%%%%%%%%%%%
\medskip
\noindent
\textbf{Step 5. Conclusion.}
%%%%%%%%%%%%%%%%%%%%%%%%%%%%%%%%%%%%%%%%%%%
 Combining \eqref{Eq:symbol-omega-b} and
\eqref{Eq:symbol-omega-N+1}, we obtain
\[
\begin{aligned}
\sum_{a=1}^{N+1}(-1)^a\widehat\Omega_a
&=
\frac{N-2}{N!}
\sum_{b=1}^{N}
\left(
(-1)^b
+
(-1)^{N+1}(-1)^{N-b}
\right)W_b
\\
&=
\frac{N-2}{N!}
\sum_{b=1}^{N}
\left(
(-1)^b+(-1)^{b+1}
\right)W_b =0.
\end{aligned}
\]

Applying the linear realisation map
$
\boldsymbol{\mathfrak s}:
\boldsymbol{\mathcal S}
\rightarrow
\Omega_{\mathrm{an}}^{N-3}(\M_{0,A})
$
and using \eqref{Eq:realization-omega-b} and
\eqref{Eq:realization-omega-N+1}, we conclude that
$
\sum_{a=1}^{N+1}
(-1)^a f_a^*(\omega_a)=0
$
on $\M_{0,A}$. This gives a second proof of the first assertion of Theorem 
\ref{Prop:Differential-Identity}.

\end{proof}

%%%%%%%%%%%%%%%%%%%%%%%%%%%%%%%%%%%%%%%%%%%%%%%%%%%%%%%%%%%%%%%%%%%%%%
\section{Proof of Theorem \ref{Coro:main}}
\label{S:Proof-Coro-Main}
%%%%%%%%%%%%%%%%%%%%%%%%%%%%%%%%%%%%%%%%%%%%%%%%%%%%%%%%%%%%%%%%%%%%%%
%Throughout this section, $n\geq2$ is even,  and
%$A=\{1,\ldots,n+3\}$. We write $\boldsymbol{\mathcal E}_n$ for the geometrically
%defined Euler relation of \cite[\S5]{M06}, to distinguish it from
%the relation $\boldsymbol E_n$ constructed here.

Throughout this section, $n\geq2$ is even and
$A=\{1,\ldots,n+3\}$, with its natural ordering. Set
\[
 X=\M_{0,A}^{\mathbb R},
 \qquad A_a=A\setminus\{a\},
 \qquad X_a=\M_{0,A_a}^{\mathbb R},
\]
and let $f_a:X\longrightarrow X_a$ be the corresponding forgetful map.
All differential forms and vector spaces in this section are real.
We denote the Euler relation of \cite[\S5]{M06} by
$\boldsymbol{\mathcal E}_n$, to distinguish it from the relation
$\boldsymbol E_n$ constructed in the present paper.
\mk 

\noindent
\textbf{Chambers and forgetful maps.}
The chambers of $X$, viewed as a configuration space modulo
$\mathrm{PGL}_2(\mathbf R)$, are indexed by cyclic orders on $A$
modulo reversal; see \cite[\S2.1.5, (2.9)--(2.10)]{M06}.
For a chamber $D\subset X$, let $D_a\subset X_a$ be the chamber
obtained by deleting $a$ from its cyclic order; this is
well-defined since deletion commutes with reversal.
Then $f_a(D)=D_a$: the inclusion is immediate, and any
configuration in $D_a$ lifts to $D$ by inserting the missing
point in the prescribed open interval between its two neighbours.

Choosing three labels in $A_a$ and normalizing their points to
$0,1,\infty$, as in \cite[\S2.1.1, (2.1)]{M06}, identifies
$f_a$ with a coordinate projection between open subsets of
real affine spaces.  
Hence $f_a|_D:D\to D_a$ is a surjective real-analytic
submersion, and its pullback on differential forms is injective.
\medskip

%%%%%%%%%%%%%%%%%%%%%%%%%%%%%%%%%%%%%%%%%%%%%%%%%%
\subsection{\bf Some spaces of abelian relations.}
%%%%%%%%%%%%%%%%%%%%%%%%%%%%%%%%%%%%%%%%%%%%%%%%%%
Let 
\begin{itemize}
\item[$-$] $\boldsymbol{AR}_C$ be the real vector space of global rational
combinatorial $(n-1)$-abelian relations of
$\boldsymbol{\mathcal W}_{0,A}^{\mathbb R}$ (cf. \cite[\S2.1.4]{M06}); 
\sk
%%%%%%%%%%%%%%%%%%%%
\item[$-$]
$\boldsymbol{AR}(D)$ be the space of real-analytic $(n-1)$-abelian
relations on any chamber $D\subset\M_{0,A}^{\mathbb R}$;
%%%%%%%%%%%%%%%%%%%%
\item[$-$] 
$\boldsymbol{AR}_{\mathrm{glob}}$ be 
the subspace of the direct sum of the $\boldsymbol{AR}(D)$'s for all chambers $D$, which is 
the space of tuples
$(\eta_a)_{a\in A}$, with $\eta_a$ a real-analytic $(n-1)$-form
on the whole of $X_a$, such that $\sum_a f_a^*(\eta_a)=0$ globally on $X$: \[
 \boldsymbol{AR}_{\mathrm{glob}}
 =\left\{
 (\eta_a)_{a\in A}\in
 \bigoplus_{a\in A}\Omega^{n-1}_{\mathrm{an}}(X_a)
 \ \middle|\ 
 \sum_{a\in A}f_a^*(\eta_a)=0\text{ on }X
 \right\}.
\]
In particular, each component is a single form on the whole target
$X_a$. Its restrictions to two source chambers with the same image
under $f_a$ must therefore agree on that common target chamber.
This compatibility is part of the definition; it would not follow
from choosing an unrelated abelian relation on each source chamber.
\end{itemize}

We identify $\boldsymbol{AR}_C$ with its restriction to any such chamber 
and thus regard it as a linear subspace of $\boldsymbol{AR}(D)$ for any $D$.  
These identifications give rise to a natural linear embedding 
$ \boldsymbol{AR}_C \hookrightarrow \boldsymbol{AR}_{\mathrm{glob}}$. Finally, by restriction, we have a linear map   $\boldsymbol{AR}_{\mathrm{glob}}
\rightarrow  \boldsymbol{AR}(D)$ for any chamber $D$, such that the composition 
$\boldsymbol{AR}_C\hookrightarrow \boldsymbol{AR}_{\mathrm{glob}} \rightarrow 
\boldsymbol{AR}(D)$ is the identification just mentioned. 
\medskip 
%\bk 
%\bk

The symmetric group $\mathfrak S_A$ acts on both 
$ \boldsymbol{AR}_C $ and $ \boldsymbol{AR}_{\mathrm{glob}}$, in a way compatible with the embedding of the former into the latter. 
 With the convention for permutations fixed in
the Introduction, its expression on tuples of pulled-back forms is
\begin{equation}
\label{Eq:section-four-relabelling}
 (\sigma\cdot R)_a
 =\sigma^*R_{\sigma^{-1}(a)}
 \end{equation}
for every $a\in A$ and every $ \sigma\in\mathfrak S_A$. 
By Theorem~\ref{Prop:Differential-Identity}, we have 
\begin{equation}
\label{Eq:section-four-E-sign}
 \boldsymbol E_n\in\boldsymbol{AR}_{\mathrm{glob}}
 \qquad\mbox{ and } \qquad 
 \sigma\cdot\boldsymbol E_n=(-1)^\sigma\boldsymbol E_n 
 \quad \forall  \sigma\in\mathfrak S_A\, .
\end{equation}

We shall use two facts about the combinatorial relations.
First, $\boldsymbol{AR}_C$ is an irreducible
$\mathfrak S_A$-module of dimension $n(n+3)/2>1$
\cite[Theorem~4.2.1]{M06}. Hence
\begin{equation}
\label{Eq:section-four-no-combinatorial-sign}
 \Big\langle\, R\in\boldsymbol{AR}_C\mid
 \sigma\cdot R=(-1)^\sigma R\text{ for every }\sigma \Big\rangle
 =\{0\}.
\end{equation}
Indeed, a nonzero such vector would span a one-dimensional invariant
subspace, contradicting irreducibility and the stated dimension.

Secondly, for distinct $i,j\in A$, set
\[
 \boldsymbol A_{\widehat{\imath\jmath}}
 = \big\langle\,R\in\boldsymbol{AR}_C\mid R_i=R_j=0\,\big\rangle.
\]
This is the one-dimensional space of relations of the quadrilateral
$(n+1)$-subweb obtained by omitting the $i$th and $j$th foliations;
see \cite[\S2.1.4 and \S4.1]{M06}. 

If $\tau=(i\,j)$, then
\begin{equation}
\label{Eq:section-four-line-action}
 \tau\cdot Q=(-1)^{n-1}Q=-Q
 %\qquad
 %\text{for every }Q\in\boldsymbol A_{\widehat{\imath\jmath}},
\end{equation}
for every  $Q\in\boldsymbol A_{\widehat{\imath\jmath}}$ 
by \cite[Lemma~4.3.2]{M06}. This is precisely where the parity
assumption will enter the transposition argument.

%%%%%%%%%%%%%%%%%%%%%%%%%%%%%%%%%%%%%%%%%%%%%%%%%%
\subsection{\bf Non-rationality of $\boldsymbol{E}_n$ and reduction to a combinatorial remainder.}
%%%%%%%%%%%%%%%%%%%%%%%%%%%%%%%%%%%%%%%%%%%%%%%%%%

By Corollary~\ref{Cor:Second-Point}, each $\omega_a$
is non-rational on every nonempty open subset of $X_a$.
For any chamber $D\subset X$, we have
$f_a(D)=D_a$, and $f_a|_D:D\to D_a$ has injective
pullback on differential forms. Since the components
of $\boldsymbol{AR}_C$ are pullbacks of rational forms
on the corresponding $X_a$, it follows that
\[
\boldsymbol E_n|_D\notin\boldsymbol{AR}_C.
\]

The rank bound and the dimension of the combinatorial space
\cite[Proposition~1.1.1 and Theorem~4.2.1]{M06} give
\[
 \dim\boldsymbol{AR} (D)\leq\frac{(n+2)(n+1)}2
 \qquad \mbox{ and } \qquad 
 \dim \boldsymbol{AR}_C=\frac{n(n+3)}2\, .
\]
Together with Theorem~\ref{Prop:Differential-Identity}, this yields
\begin{equation}
\label{Eq:Euler-affine-splitting}
\boldsymbol{AR} (D)=\boldsymbol{AR}_C \oplus\R\,\boldsymbol E_n|_D\, .
\end{equation}

%%%%%%%%%%%%%%%%%%%%%%%%%%%%%%%%%%%%%%%%%%%%%%%%
\subsection{Uniqueness of a global relation with the sign character}
%%%%%%%%%%%%%%%%%%%%%%%%%%%%%%%%%%%%%%%%%%%%%%%%
\begin{prop}%%%%%%%%%%%%%%%%%%%%%%%%%%%%%%%%%%%%%%%%%
\label{P:global-sign-uniqueness}%%%%%%%%%%%%%%%%%%%%%%%%%%%%%%
%%%%%%%%%%%%%%%%%%%%%%%%%%%%%%%%%%%%%%%%%%%%%%%%
The subspace of $\boldsymbol{AR}_{\mathrm{glob}}$ transforming
according to the sign character is exactly
$\R\boldsymbol E_n$. Equivalently, if
$\Psi\in\boldsymbol{AR}_{\mathrm{glob}}$ satisfies
\[
 \sigma\cdot\Psi=(-1)^\sigma\Psi
 \qquad\text{for every }\sigma\in\mathfrak S_A,
\]
there is a single constant $c\in\R$ such that
$\Psi=c\boldsymbol E_n$ on the whole of $X$.
\end{prop}

\begin{proof}
For every chamber $D$, the decomposition
\eqref{Eq:Euler-affine-splitting} gives unique elements
$c_D\in\R$ and $R_D\in\boldsymbol{AR}_C$ such that
\begin{equation}
\label{Eq:section-four-fixed-chamber}
 \left.\Psi\right|_D
 =c_D\,\left.\boldsymbol E_n\right|_D
 +\left.R_D\right|_D.
\end{equation}
Each $R_D$ is globally defined, but both $c_D$ and the choice of
$R_D$ may depend on $D$. We first prove their independence.

Let $D$ be any chamber, let $i,j$ be consecutive labels in a cyclic
order representing it, and put $\tau=(i\,j)$ and $D'=\tau(D)$.
Since $\tau^2=1$, applying the relabelling action to
\eqref{Eq:section-four-fixed-chamber} gives an equality on $D'$.
Using $\tau\cdot\Psi=-\Psi$ and
$\tau\cdot\boldsymbol E_n=-\boldsymbol E_n$, we obtain
\[
 \Psi|_{D'}=c_D\boldsymbol E_n|_{D'}
             +(-\tau\cdot R_D)|_{D'}.
\]
The relation $-\tau\cdot R_D$ belongs to $\boldsymbol{AR}_C$.
Comparing this equality with the unique decomposition on $D'$ gives
\begin{equation}
\label{Eq:section-four-transposed-remainder}
 c_{D'}=c_D,
 \qquad R_{D'}=-\tau\cdot R_D.
\end{equation}
The second equality is an equality of global rational relations,
because restriction of $\boldsymbol{AR}_C$ to $D'$ is injective.
It does not yet assert that $R_{D'}=R_D$.

Deleting either member of the adjacent pair leaves the same cyclic
order for $D$ and $D'$. The surjectivity onto target chambers proved
above therefore gives
\begin{equation}
\label{Eq:section-four-common-targets}
 f_i(D)=f_i(D'),\qquad f_j(D)=f_j(D').
\end{equation}
We compare target forms on these common open sets. Write
\[
 \Psi_a=f_a^*\psi_a,\qquad
 (\boldsymbol E_n)_a=f_a^*e_a,\qquad
 (R_D)_a=f_a^*\rho_{D,a},
\]
where $e_a=(-1)^a\omega_{A_a}$, the forms $\psi_a,e_a$ are
real analytic on the whole of $X_a$, and $\rho_{D,a}$ is rational.
The decompositions on $D$ and $D'$, together with $c_{D'}=c_D$
and injectivity of pullback, give
\[
 \rho_{D,i}=\psi_i-c_De_i=\rho_{D',i}
 %\quad\text{on }f_i(D)=f_i(D'),
\]
on $f_i(D)=f_i(D')$ and the same statement holds with $j$ in place of $i$.
Since the differences of these rational forms vanish on nonempty
real open sets, they vanish identically. Thus
\begin{equation}
\label{Eq:section-four-difference-line}
 \Delta_{D,\tau}:=R_{D'}-R_D=-\tau\cdot R_D-R_D
 \ \in\ \boldsymbol A_{\widehat{\imath\jmath}}.
\end{equation}

On the other hand, the involutivity of $\tau$ gives
\[
 \tau\cdot\Delta_{D,\tau}
 =-\tau^2\cdot R_D-\tau\cdot R_D
 =-R_D-\tau\cdot R_D
 =\Delta_{D,\tau}.
\]
Since $\tau\cdot\Delta_{D,\tau}=-\Delta_{D,\tau}$ by \eqref{Eq:section-four-line-action}, it follows that $\Delta_{D,\tau}=0$, and therefore
%%%%%%%%%%%%%%%%%%%%%%%%%%%%%%%%%%%%%%%%%%%%%%%
\begin{equation}
\label{Eq:section-four-adjacent-sign}
 R_{D'}=R_D\, .
% \qquad \tau\cdot R_D=-R_D.
\end{equation}
Thus both $c_D$ and $R_D$ are unchanged when two consecutive
labels in the cyclic order are exchanged.

These exchanges connect all chambers. Indeed, choose linear
representatives of the cyclic orders of two chambers. One is obtained
from the other by successive exchanges of neighbouring entries;
each induces an allowed exchange of cyclic orders, up to reversal.
Thus there are a single constant $c\in\R$
and a single rational relation $R\in\boldsymbol{AR}_C$ such that,  for every chamber $D$, one has 
\[
 c_D=c,\qquad R_D=R\, .
 % \qquad\text{for every chamber }D.
\]
Equation~\eqref{Eq:section-four-fixed-chamber} now holds with
these same $c$ and $R$ on every chamber, hence on the whole of $X$.
Surjectivity of every $f_a:X\longrightarrow X_a$ then gives
the following equality in $\boldsymbol{AR}_{\mathrm{glob}}$: 
\begin{equation}
\label{Eq:section-four-global-remainder}
 \Psi=c\boldsymbol E_n+R\, .
% \qquad\text{in }\boldsymbol{AR}_{\mathrm{glob}}.
\end{equation}

It remains to show that this common remainder is zero.
Applying any $\sigma\in\mathfrak S_A$ to
\eqref{Eq:section-four-global-remainder} and using the sign
equivariance of $\Psi$ and $\boldsymbol E_n$ gives
\[
 \sigma\cdot R
 =\sigma\cdot\Psi-c\,\sigma\cdot\boldsymbol E_n
 =(-1)^\sigma(\Psi-c\boldsymbol E_n)
 =(-1)^\sigma R.
\]
Since $R\in\boldsymbol{AR}_C$, equation
\eqref{Eq:section-four-no-combinatorial-sign} implies $R=0$.
Consequently $\Psi=c\boldsymbol E_n$ globally.
The reverse inclusion in the proposition follows from
\eqref{Eq:section-four-E-sign}.
\end{proof}

\subsection{Application to the geometric Euler relation}

For an ordered set $B$ of cardinality $n+2$, let $G_B^\circ$ be
the open subset of $G_2(\R^B)$ on which all Pl\"ucker coordinates
are nonzero. The diagonal torus
\[
 T_B^+=(\R_{>0})^B/_{\R_{>0}}\,,\quad
 K_B=\{\pm1\}^B/_{\{\pm1\}}
 \quad \mbox{ and } \quad  T_B=T_B^+\times K_B=(\R^*)^B/_{\R^*},\qquad
\]
acts on $G_B^\circ$, and
$G_B^\circ/T_B=X_B=\M_{0,B}^{\mathbb R}$.
Let $P_{1,B}$ be the $O(\R^B)$-invariant representative of the
first Pontryagin class of the tautological rank-$2$ bundle, as in
\cite[\S5.7.1]{M06}, and put $\alpha_B=P_{1,B}^{\,n/2}$.
We orient $T_B^+$ by the orientation of its Lie algebra
$\R^B/\R(1,\ldots,1)$ induced by the order on $B$. On the oriented Grassmannian
$G_2^{or}(\R^B)$, the pull-back of $P_{1,B}$ is $E_B\wedge E_B$, where $E_B$ is the
$\mathrm{ SO}(\R^B)$-invariant representative of the Euler class of the tautological bundle;
thus $\alpha_B$ pulls back to $E_B^{\,n}$, and the results of \cite[\S5.3]{M06}, stated for
$E^n$ on the positive part, apply to $\alpha_B$ (see \cite[\S5.7.1]{M06}).
\medskip

\vspace{-0.4cm} 
The following lemma is essentially contained in \cite[Lemma~5.7.1 and Proposition~5.7.2]{M06},
where Euler's relation is considered chamber by chamber. What is new here is that the
construction yields forms defined on the whole of $X_B$, with a normalization which is coherent
across all chambers; this is what allows us to regard Euler's relation as an element of
$\boldsymbol{AR}_{\mathrm{glob}}$ below.

%%%%%%%%%%%%%%%%%%%%%%%%%%%%%%%%%%%%%%%%%%%%%
\begin{lem}\label{L:Euler-global-descent}%%%%%%%%%%%%%%%%%%%%%%%
%%%%%%%%%%%%%%%%%%%%%%%%%%%%%%%%%%%%%%%%%%%%%
Fibre integration of $\alpha_B$ along the oriented $T_B^+$-orbits
induces a well-defined real-analytic $(n-1)$-form $\theta_B$
on the entire real moduli space $X_B$. This form is independent
of all choices of nonzero lifts of the marked projective points.
Moreover, under an order-changing relabelling of $B$, it
transforms by the signature of that relabelling.
%%%%%%%%%%%%%%%%%%%%%%%%%%%%%%%%%%%%%%%%%%%%%
\end{lem}%%%%%%%%%%%%%%%%%%%%%%%%%%%%%%%%%%%%%%%%
%%%%%%%%%%%%%%%%%%%%%%%%%%%%%%%%%%%%%%%%%%%%%

\begin{proof}
Set $\widehat X_B=G_B^\circ/T_B^+$ and denote by
$q_B:G_B^\circ\to\widehat X_B$ the quotient map.
The $T_B^+$-action is free on $G_B^\circ$: an element fixing
a plane induces a projective transformation fixing all the
marked points, hence is scalar since there are at least three
distinct marked points. Thus $q_B$ is a smooth bundle with
oriented fibres of dimension $n+1$. Its fibre integral $
 \widehat\theta_B=(q_B)_*\alpha_B
$ 
has degree $2n-(n+1)=n-1$.

The convergence and real-analytic dependence of this integral
are those established in \cite[\S5.3]{M06} (cf. Proposition
5.3.1(2) and Proposition 5.3.2 therein)  with the dimension parameter
there replaced by $n-1$. Indeed, on a positive chamber the
orbits are parametrized by the interior of the hypersimplex
$$
\Delta_{B}=\bigg\{\, (t_b)_{b\in B}\in [0,1]^B \, \Big\lvert \, \sum_{b\in B} t_b =2\, \bigg\}\, 
$$
and the compactification used there provides fibre integration
with smooth boundary behaviour.

Every other chamber is carried to a positive chamber by a
relabelling of the indices and suitable changes of signs of the
coordinates. These transformations lift to orthogonal transformations
of $\R^B$, which normalize the action of the diagonal torus and
preserve $P_{1,B}$, hence $\alpha_B$
(see \cite[Corollary~3.25]{GelfandMacPherson} and \cite[\S5.7.1]{M06}).
A relabelling multiplies the orientation of the $T_B^+$-orbits
by its signature, whereas changes of signs preserve this
orientation. Hence the convergence and real-analyticity
statements above hold on every chamber of $\widehat X_B$.
%Every other chamber is
%carried to a positive chamber by a relabelling of the indices;
%possible choices of signs of the lifts act by orthogonal
%diagonal matrices. Since $\alpha_B$ is invariant under these
%orthogonal transformations, the same conclusions hold on
%every chamber of $\widehat X_B$.
\sk

The remaining quotient $\widehat X_B\to X_B$ is the finite
covering with deck group $K_B$. For $\varepsilon\in K_B$,
let $D_\varepsilon$ be a diagonal matrix representing it.
Because $D_\varepsilon$ is orthogonal,
$D_\varepsilon^*\alpha_B=\alpha_B$. Moreover,
$D_\varepsilon$ commutes with $T_B^+$, hence it induces
the identity automorphism of the Lie algebra
$\R^B/\R(1,\ldots,1)$ and preserves the chosen orientation
of every $T_B^+$-orbit. Naturality of integration along
oriented fibres gives
$ D_\varepsilon^*\widehat\theta_B=\widehat\theta_B$. 
Consequently $\widehat\theta_B$ descends uniquely to a
real-analytic form $\theta_B$ on $X_B$.
This establishes the independence of the choice of lifts:
two systems of nonzero lifts differ, modulo the change of
basis already absorbed by passage to the Grassmannian,
by diagonal rescalings, whose positive parts lie in $T_B^+$
and whose signs lie in $K_B$.
\sk

As noted above, a bijection of ordered label sets induces an
orthogonal permutation of the coordinate spaces. It
preserves the forms $\alpha_B$ and acts on the oriented
quotient Lie algebra by its signature. The corresponding
pullback of $\theta_B$ is therefore multiplied by that
signature, as asserted.
\end{proof}
\sk

We apply this construction to every ordered set
$A_a=A\setminus\{a\}$, using the same normalization of the
Pontryagin forms and the quotient orientations induced by
the respective label orders. The Stokes argument of
Proposition 5.3.1(1) and Proposition 5.3.5(1)--(2) in 
\cite[\S5.3]{M06}
then applies on each chamber of $X=\M_{0,A}^{\mathbb R}$.
More precisely, the integral of $P_{1,A}^{\,n/2}$ over the
positive diagonal orbits vanishes by Proposition 5.3.1(1),
and the boundary of the hypersimplex is the alternating sum
of its facets. The facets $\{t_a=0\}\cap \Delta_A\simeq \Delta_{A_a}$ yield, with their induced
orientations, the pullbacks $f_a^*\theta_{A_a}$ with
coefficients $(-1)^a$;   the facets $\{t_a=1\}\cap\Delta_A$ contribute zero: on them the plane
 contains the $a$-th basis vector, so these facets are mapped into a submanifold of dimension
$n+1<2n=\deg P_{1,A}^{\,n/2}$, on which $P_{1,A}^{\,n/2}$ restricts to zero.  Relabelling a chamber to the positive one, with suitable changes
of signs, preserves this identity, since the permutation sign
occurs both in the orbit orientation and in the ordered facet
signs, while changes of signs preserve both (see the proof
of Lemma~\ref{L:Euler-global-descent}).
 We therefore obtain, globally,
\[
 \sum_{a\in A}(-1)^a f_a^*\theta_{A_a}=0.
\]
By Lemma~\ref{L:Euler-global-descent}, all the components
$\theta_{A_a}$ are defined on their whole targets
$X_a$, not merely on the images of individual source
chambers. Hence, with one fixed overall normalization,
$\boldsymbol{\mathcal E}_n$
is represented in $\boldsymbol{AR}_{\mathrm{glob}}$ by
$\bigl((-1)^a\theta_{A_a}\bigr)_{a\in A}$.

%%%%%%% gogagoga

This relation is nonzero by
\cite[Proposition~5.3.5(3)]{M06}, and it follows from \cite[Proposition~5.7.2(1)]{M06} that 
for every  $\sigma\in\mathfrak S_A$, we have 
\[
 \sigma\cdot\boldsymbol{\mathcal E}_n
 =(-1)^\sigma\boldsymbol{\mathcal E}_n
 %\qquad \forall\, \sigma\in\mathfrak S_A
 \,.
\]

Proposition~\ref{P:global-sign-uniqueness} therefore gives that 
\[
 \boldsymbol{\mathcal E}_n=c\boldsymbol E_n
% ,\qquad c\in\R\setminus\{0\},
\]
for some constant $c\in\R\setminus\{0\}$, 
which proves Theorem~\ref{Coro:main}.
\mk

%%%%%%%%%%%%%%%%%%%%%%%%%%%%%%%%%%%%%%%%%
%%%%%%%%%%%%%%%%%%%%%%%%%%%%%%%%%%%%%%%%%
%%%%%%%%%%%%%%%%%%%%%%%%%%%%%%%%%%%%%%%%%
\section{\bf Comparison with the explicit formula of \cite{M06}.}
\label{S:Link-explicit}
%%%%%%%%%%%%%%%%%%%%%%%%%%%%%%%%%%%%%%%%%
%%%%%%%%%%%%%%%%%%%%%%%%%%%%%%%%%%%%%%%%%

We now relate the differential form constructed above to the explicit formula
conjectured in \cite[\S5.9.2]{M06}. Throughout this section, $n\geq 2$ is even,
and we set
\[
N=n+2,
\qquad
m=n+1=N-1.
\]
Thus the differential form $\omega_B$ associated with an ordered set $B$ of
cardinality $N$ has degree $N-3=n-1$, as do the components of the Euler
abelian relation of $\boldsymbol{\mathcal W}_{0,n+3}$.

%%%%%%%%%%%%%%%%%%%%%%%%%%%%%%%%%%%%%%%%%
\subsection{The formula of \cite{M06}}
%%%%%%%%%%%%%%%%%%%%%%%%%%%%%%%%%%%%%%%%%
We use the affine coordinates of \cite[\S5.9.2]{M06}. 
Let
$
x'=(x_2,\ldots,x_n)
$ and, for notational convenience, set
$
x_{n+1}=0$ and $x_{n+2}=1$.
\sk 

\vspace{-0.35cm}
For an ordered $r$-tuple $z=(z_1,\ldots,z_r)$, we put ${\rm M}_r(z)=\prod_{1\leq i<j\leq r}(z_i-z_j)$.
%\bk 

Let
$
\tilde x'
=
(x_2,\ldots,x_n,x_{n+1},x_{n+2})
=
(x_2,\ldots,x_n,0,1)$.
We set
$$
M_{0,n+2}(x')={\rm M}_{n+1}(\tilde x')
=-\prod_{i=2}^{n}x_i(x_i-1)\prod_{2\leq i<j\leq n}(x_i-x_j)\, .
$$
 For $2\leq i<j\leq n+2$, we denote by
$ \widetilde x'_{\widehat{\imath\jmath}}
$ the ordered $(n-1)$-tuple obtained from $\widetilde x'$ by deleting
$x_i$ and $x_j$: 
\[
\tilde x'_{\widehat{\imath\jmath}}
=
\big(
x_2,\ldots,x_{i-1},
\widehat{x_i},
x_{i+1},\ldots,x_{j-1},
\widehat{x_j},
x_{j+1},\ldots,x_{n+2}
\big)\, . 
\]

 The explicit formula conjectured in \cite[(5.65)]{M06} involves the function 
\begin{equation}
\label{Eq:epsilon-M06}
\varepsilon_{n-1}(x')
=
\frac{1}{M_{0,n+2}(x')}
\sum_{2\leq i<j\leq n+2}
(-1)^{j-i}\,
{\rm M}_{n-1}\big(
\tilde x'_{\widehat{\imath\jmath}}
 %x'_{\widehat{\imath\jmath}}
\big)
(x_i-x_j)^{n-1}
\Log\big|\,x_i-x_j\,\big| \, .
\end{equation}
The term corresponding to $(i,j)=(n+1,n+2)$ is zero, so that
\eqref{Eq:epsilon-M06} is the same formula as the one written in
\cite[\S5.9.2]{M06}.\footnote{In \cite[\S5.9.2]{M06}, $M_{0,n+3}(x)$ is given by two expressions,
${\rm M}_{n+2}(\tilde x)$ and $\prod_{i=1}^n x_i(x_i-1)\prod_{1\leq k<\ell\leq n}(x_k-x_\ell)$,
which differ by the factor $x_{n+1}-x_{n+2}=-1$. The first one is the convention for which
\cite[(5.65)]{M06} gives back \cite[(5.61)]{M06} when $n=4$; it is the one used here
(this corrects an inconsistency of sign in \cite[\S5.9.2]{M06}).} 
To rewrite \eqref{Eq:epsilon-M06} in a symmetric form, 
set
\[
y=(y_1,\ldots,y_m)
=(x_2,\ldots,x_n,0,1),
\qquad
y_{k\ell}=y_k-y_\ell,
\]
and let  $M_y(t)=\prod_{r=1}^{m}(t-y_r)$ and 
 $M_k=M_y'(y_k)
=
\prod_{r\neq k}(y_k-y_r)$.
%\[
%M_y(t)=\prod_{r=1}^{m}(t-y_r),
%\qquad
%M_k=M_y'(y_k)
%=
%\prod_{r\neq k}(y_k-y_r).
%\]

For distinct $k,\ell\in\{1,\ldots,m\}$, define
\begin{equation}
\label{Eq:C-kl}
K_{k\ell}=K_{k\ell}(y)
=
\frac{(y_k-y_\ell)^n}{M_kM_\ell}.
\end{equation}
Since $n$ is even, $K_{k\ell}=K_{\ell k}$.

\begin{lem}
\label{L:epsilon-C}
With the notation above,
\begin{equation}
\label{Eq:epsilon-C}
\varepsilon_{n-1}(x')
=
\sum_{1\leq k<\ell\leq m}
K_{k\ell}\Log|y_{k\ell}|.
\end{equation}
\end{lem}

\begin{proof}
Setting $
\Delta(y)=\prod_{1\leq r<s\leq m}(y_r-y_s)$, 
since $y=\tilde x'$, we have
\begin{equation}
\label{Eq:Delta-M0}
\Delta(y)=M_{0,n+2}(x').
\end{equation}

For $k<\ell$, let $\Delta_{\widehat{k\ell}}(y)$ denote the
Vandermonde product of the ordered tuple obtained from $y$ by deleting
$y_k$ and $y_\ell$.  We have 
\[
\frac{\Delta(y)}{\Delta_{\widehat{k\ell}}(y)}
=
(-1)^{\ell-k-1}y_{k\ell}
\prod_{s\neq k,\ell}(y_k-y_s)(y_\ell-y_s)
=
(-1)^{\ell-k}\frac{M_kM_\ell}{y_{k\ell}},
\]
from which it follows that 
%Moreover, a direct separation in the Vandermonde product of the factors
%involving $y_k$ or $y_\ell$ gives
\begin{equation}
\label{Eq:Vandermonde-minor}
\frac{\Delta_{\widehat{k\ell}}(y)}{\Delta(y)}
=
(-1)^{\ell-k}\frac{y_{k\ell}}{M_kM_\ell}\, .
\end{equation}

Now put $k=i-1$ and $\ell=j-1$. Then
$j-i=\ell-k$ and
$
{\rm M}_{n-1}(\tilde x'_{\widehat{\imath\jmath}})
%{\rm M}_{n-1}(x'_{\widehat{\imath\jmath}})
=\Delta_{\widehat{k\ell}}(y)$. Using
\eqref{Eq:Delta-M0} and \eqref{Eq:Vandermonde-minor}, it follows that the coefficient of
$\Log|y_{k\ell}|$ in \eqref{Eq:epsilon-M06} is
\[
(-1)^{\ell-k}
\frac{\Delta_{\widehat{k\ell}}(y)}{\Delta(y)}
(y_{k\ell})^{n-1}
=
\frac{(y_{k\ell})^n}{M_kM_\ell}
=K_{k\ell}\, . \qedhere
\]
%This proves \eqref{Eq:epsilon-C}.
\end{proof}

%\begin{rem}
%The coefficients $K_{k\ell}$ satisfy
%\begin{equation}
%\label{Eq:sum-C-zero}
%\sum_{1\leq k<\ell\leq m}K_{k\ell}=0.
%\end{equation}
%Indeed, if
%$S_r=\sum_{k=1}^m y_k^r/M_k$, Lagrange interpolation gives
%$S_0=\cdots=S_{m-2}=0$ and $S_{m-1}=1$. Since $n=m-1$ is even,
%\[
%2\sum_{k<\ell}K_{k\ell}
%=
%\sum_{k,\ell=1}^m\frac{(y_k-y_\ell)^{m-1}}{M_kM_\ell}
%=
%\sum_{r=0}^{m-1}
%(-1)^r{m-1\choose r}S_{m-1-r}S_r=0.
%\]
%This identity is not needed in the proof below, but it is consistent with the
%projective nature of the formula: multiplying all differences $y_{k\ell}$ by
%the same nonzero constant does not change the right-hand side of
%\eqref{Eq:epsilon-C}.
%\end{rem}

%%%%%%%%%%%%%%%%%%%%%%%%%%%%%%%%%%%%%%%%%
\subsection{A differential-form rewriting}
%%%%%%%%%%%%%%%%%%%%%%%%%%%%%%%%%%%%%%%%%
We regard the functions $K_{k\ell}$ of \eqref{Eq:C-kl} as rational
functions on 
$$\mathcal Y_m=
\big\{(y_s)_{s=1}^m \in\C^m\,;\,y_k\neq y_\ell
\text{ for }k\neq\ell\big\}$$ and we set 
%$T=\sum_{r=1}^m\frac{\partial}{\partial y_r}$, 
%$E=\sum_{r=1}^m y_r\frac{\partial}{\partial y_r}$ and 
%%%%%%%%%%%%%%%%%%%%%%%%%%%%%%%%%%%%%%%%%
\begin{equation}%%%%%%%%%%%%%%%%%%%%%%%%%%%%%%%%
\label{Eq:Xi}%%%%%%%%%%%%%%%%%%%%%%%%%%%%%%%%%%
%%%%%%%%%%%%%%%%%%%%%%%%%%%%%%%%%%%%%%%%%
T=\sum_{r=1}^m\frac{\partial}{\partial y_r}\, , \qquad 
E=\sum_{r=1}^m y_r\frac{\partial}{\partial y_r} 
%T=\sum_{r=1}^m \partial_{y_r}\, , \qquad 
%E=\sum_{r=1}^m y_r\partial_{y_r} 
\qquad \mbox{and}\qquad 
\Xi=\iota_E\iota_T\Big( \wedge_{s=1}^m \dd y_s \Big) \, .
%\dd y_1\wedge\cdots\wedge\dd y_m \big) \, . 
%%%%%%%%%%%%%%%%%%%%%%%%%%%%%%%%%%%%%%%%%
\end{equation}%%%%%%%%%%%%%%%%%%%%%%%%%%%%%%%%%
%%%%%%%%%%%%%%%%%%%%%%%%%%%%%%%%%%%%%%%%%
%\[
%\mathcal Y_m=
%\big\{(y_1,\ldots,y_m)\in\C^m\,;\,y_k\neq y_\ell
%\text{ for }k\neq\ell\big\}.
%\] 

For $1\leq k<\ell\leq m$, define on $\mathcal Y_m$
\begin{equation}
\label{Eq:Theta-kl}
\Theta_{k\ell}
=
\bigwedge_{\substack{1\leq s\leq m\\s\notin\{k,\ell\}}}
\dd\Log\left(\frac{y_s-y_k}{y_s-y_\ell}\right),
\end{equation}
where the factors are ordered according to increasing $s$.

\begin{lem}
\label{L:wedge-C}
For every $1\leq k<\ell\leq m$, one has on $\mathcal Y_m$
\begin{equation}
\label{Eq:Theta-C}
\Theta_{k\ell}
=
(-1)^{\ell-k-1}K_{k\ell}\,\Xi.
\end{equation}
\end{lem}

\begin{proof}
We use the same argument as in the computation leading to
\eqref{Eq:Delta-b-Theta}. Indeed, 
let $W\subset\Omega^1_{\C(y)}$ denote the
vector space of rational $1$-forms in the variables $y_1,\ldots,y_m$ annihilating both $T$ and $E$. 
Then $ W$ has dimension $m-2$ and both $\Theta_{k\ell}$ and $\Xi$ belong to the
one-dimensional space $\bigwedge^{m-2}W$, and are therefore
proportional over $\mathbf C(y)$. 

To determine the proportionality factor, we evaluate these forms on 
the $(m-2)$-tuple
 of rational vector fields 
$\xi_{k\ell}
=\big( \,
{\partial}/{\partial y_s}\big)_{s\notin\{k,\ell\}},
$
where the ${\partial}/{\partial y_s}$'s are ordered by increasing $s \notin\{k,\ell\}$. 
A direct calculation gives
\[
\Xi\big(\xi_{k\ell}\big)
=
(-1)^{\ell-k}y_{k\ell}
\qquad \mbox{ and } \qquad 
\Theta_{k\ell}\big(\xi_{k\ell}\big)
=
\frac{(y_{k\ell})^{m-2}}
{\displaystyle\prod_{s\neq k,\ell}y_{sk}y_{s\ell}}
\, .
\]
Since
$
M_kM_\ell
=
-(y_{k\ell})^2
\prod_{s\neq k,\ell}y_{sk}y_{s\ell}$,
and $n=m-1$, it follows that
\[
\frac{\Theta_{k\ell}(\xi_{k\ell})}
     {\Xi(\xi_{k\ell})}
=
(-1)^{\ell-k-1}
\frac{(y_{k\ell})^n}{M_kM_\ell}
=
(-1)^{\ell-k-1}K_{k\ell},
\]
which proves \eqref{Eq:Theta-C}.
\end{proof}

For $k<\ell$, set
\begin{equation}
\label{Eq:Omega-kl}
\Omega_{k\ell}
=
(-1)^{\ell-k-1}\Log|y_{k\ell}|\,\Theta_{k\ell}.
\end{equation}
Lemma~\ref{L:wedge-C} immediately gives
\begin{equation}
\label{Eq:Omega-C}
\Omega_{k\ell}
=
K_{k\ell}\Log|y_{k\ell}|\,\Xi.
\end{equation}

We now return to the normalization used above. We set
$$ U=\left\{\, 
(y_i)_{i=1}^{m-2} \in\C^{m-2}
\,\,\vert\,\,
y_i\neq0,1,\, y_i\neq y_j\ (i\neq j)\, 
\right\}\, , 
$$
let 
\[
\jmath:
%\C^{m-2}\supset 
U\longrightarrow\mathcal Y_m,
\qquad
(y_1,\ldots,y_{m-2})
\longmapsto
(y_1,\ldots,y_{m-2},0,1)
\]
be the corresponding affine section, and put
$\boldsymbol\eta
=\dd y_1\wedge\cdots\wedge\dd y_{m-2}
=\dd x_2\wedge\cdots\wedge\dd x_n$.
A direct evaluation of \eqref{Eq:Xi} gives
$$\jmath^*\Xi(\partial_{y_1},\ldots,\partial_{y_{m-2}}\big)=
\big(
\dd y_1\wedge\cdots\wedge\dd y_m
\big)
(T,E,\partial_{y_1},\ldots,\partial_{y_{m-2}}\big)= y_m-y_{m-1}=1$$ hence we have 
$\jmath^*\Xi=\boldsymbol\eta$. 
%\begin{equation}
%\label{Eq:Xi-eta}
%\jmath^*\Xi=\boldsymbol\eta.
%\end{equation}
\sk

 Combining this with \eqref{Eq:Omega-C} 
%and  \eqref{Eq:Xi-eta},
 and Lemma~\ref{L:epsilon-C}, we obtain
\begin{equation}
\label{Eq:epsilon-sum-Omega}
\varepsilon_{n-1}(x')\,\boldsymbol\eta
= 
\sum_{1\leq k<\ell\leq m}\jmath^*\Big(\Omega_{k\ell}\Big)\, .
\end{equation}

%%%%%%%%%%%%%%%%%%%%%%%%%%%%%%%%%%%%%%%%%
\subsection{Antisymmetrization}
%%%%%%%%%%%%%%%%%%%%%%%%%%%%%%%%%%%%%%%%%
In particular, on $\mathcal Y_m$, we have 
\begin{equation}
\label{Eq:Omega-12}
\Omega_{12}
=
\Log \big|\,y_1-y_2\,\big|
\bigwedge_{s=3}^{m}
\dd\Log\left(\frac{y_s-y_1}{y_s-y_2}\right).
\end{equation}
The symmetric group $\mathfrak S_m$ acts on $\mathcal Y_m$ by permutation
of the coordinates.

%zzzz

\begin{lem}
\label{L:Alt-Omega}
On  $\mathcal Y_m$, one has 
%%%%%%%%%%%%%%%%%%%%%%%%%%%%%%%%%%%%%%%%%
\begin{equation}%%%%%%%%%%%%%%%%%%%%%%%%%%%%%%%%
\label{Eq:Alt-Omega-unnormalized}%%%%%%%%%%%%%%%%%%%%%%
%%%%%%%%%%%%%%%%%%%%%%%%%%%%%%%%%%%%%%%%%
\sum_{\sigma\in\mathfrak S_m}
(-1)^\sigma\sigma^*\big( \Omega_{12}\big)
=
2(m-2)! \hspace{-0.25cm}\sum_{1\leq k<\ell\leq m}\Omega_{k\ell}.
\end{equation}
Consequently, pulling-back by  $\jmath$, we get 
\begin{equation}
\label{Eq:Alt-Omega}
\jmath^*\Big( \Alt_m(\Omega_{12})\Big)
=
\frac{2}{m(m-1)}\,
\varepsilon_{n-1}(x')\,\boldsymbol\eta\, .
\end{equation}
%%%%%%%%%%%%%%%%%%%%%%%%%%%%%%%%%%%%%%%%%
\end{lem}%%%%%%%%%%%%%%%%%%%%%%%%%%%%%%%%%%%%
%%%%%%%%%%%%%%%%%%%%%%%%%%%%%%%%%%%%%%%%%
\begin{proof}
Fix $k<\ell$, and let
$\pi_{k\ell}\in\mathfrak S_m$ be the permutation such that 
$\big(\pi_{k\ell}(i)\big)_{i=1}^m=
(k,\ell,s_1,\ldots,s_{m-2})$ 
with $s_1<\cdots<s_{m-2}$. Its signature is
\[
(-1)^{\pi_{k\ell}}
=(-1)^{k+\ell-3}
=(-1)^{\ell-k-1}.
\]
It follows directly from \eqref{Eq:Omega-kl} and \eqref{Eq:Omega-12} that
$(-1)^{\pi_{k\ell}}\pi_{k\ell}^*\big(\Omega_{12}\big) 
=\Omega_{k\ell}$. 
%\begin{equation}
%\label{Eq:canonical-pair}
%(-1)^{\pi_{k\ell}}\pi_{k\ell}^*\Omega_{12}
%=\Omega_{k\ell}.
%\end{equation}
\sk

If $\sigma\in \mathfrak S_m$ is such that $\sigma(1)=k$ and $\sigma(2)=\ell$, then $\sigma$ is obtained from
$\pi_{k\ell}$ by permuting the last $m-2$ entries. Such a permutation changes
the sign of the wedge product by exactly the same sign as it changes the
signature of $\sigma$. Hence every one of the $(m-2)!$ corresponding signed
terms $(-1)^\sigma\sigma^*(\Omega_{12})$ is equal to $\Omega_{k\ell}$.

The same is true when $\sigma(1)=\ell$ and $\sigma(2)=k$. Indeed, exchanging
the first two entries reverses every logarithmic ratio in
\eqref{Eq:Omega-12}, and hence multiplies the wedge product by
$(-1)^{m-2}=-1$, because $m-2=n-1$ is odd. At the same time it reverses the
signature of the permutation. Thus the signed term is unchanged. We obtain
$2(m-2)!$ copies of $\Omega_{k\ell}$ for each unordered pair $\{k,\ell\}$,
which proves \eqref{Eq:Alt-Omega-unnormalized}.
\sk 

Dividing by $m!$ gives
$
\Alt_m(\Omega_{12})
=
\frac{2}{m(m-1)}
\sum_{k<\ell}\Omega_{k\ell}$.
 After pulling back by $\jmath$, formula \eqref{Eq:Alt-Omega} follows from
\eqref{Eq:epsilon-sum-Omega}.
\end{proof}

%%%%%%%%%%%%%%%%%%%%%%%%%%%%%%%%%%%%%%%%%
\subsection{Identification with the form $\omega_B$}
%%%%%%%%%%%%%%%%%%%%%%%%%%%%%%%%%%%%%%%%%
\label{SS:Identification}
Let $N=n+2=m+1$ and 
$B=\{1,\ldots,N\}$,
and consider the standard affine chart of $\M_{0,B}$ defined by
$
p_{N-2}=0$,
$p_{N-1}=1$ and 
$
p_N=\infty$.

We lift the normalized projective configuration to
$\widetilde U_B$ by choosing 
\[
v_r=(y_r,1)\quad(1\leq r\leq m),
\qquad
v_N=(1,0),
\]
with $y_{m-1}=0$ and $y_m=1$ 
so that $p_r=[v_r]$. This yields a  section of
$q_B:\widetilde U_B\to\M_{0,B}$
over the affine chart under consideration.

From \eqref{Eq:omega-gamma}, we have 
$
\widetilde\omega_B
=(N-2)\Alt_B(\gamma_B)
$, with $
\gamma_B
=
\Log|[12]|
\bigwedge_{r=3}^{N-1}
\dd\Log\left(\frac{[1r]}{[2r]}\right)\, .$ 
For the above choice of lifts, $\gamma_B$ is precisely the form
$\Omega_{12}$ defined in \eqref{Eq:Omega-12}. 
Moreover, in the antisymmetrization of $\gamma_B$, every term indexed by a
permutation $\sigma\in\mathfrak S_N$ with $\sigma(N)\neq N$ vanishes on this
section. Indeed, if the label $N$ occurs in one of the first two positions,
the logarithmic coefficient is $\Log|[Nj]|=\Log 1=0$; if it occurs in one
of the positions $3,\ldots,N-1$, one of the logarithmic differentials is
$
\dd\Log\big({[iN]}/{[jN]}\big)=0$.
Consequently, only the permutations fixing $N$ contribute, and therefore
%%%%%%%%%%%%%%%%%%%%%%%%%%%%%%%%%%%%%%%%
\begin{align}%%%%%%%%%%%%%%%%%%%%%%%%%%%%%%%%%
%%%%%%%%%%%%%%%%%%%%%%%%%%%%%%%%%%%%%%%%
\omega_B
&=
\frac{N-2}{N!} \jmath^*\left(
\sum_{\sigma\in\mathfrak S_m}
(-1)^\sigma\sigma^*\big( \Omega_{12}\big)\right)
=
\frac{N-2}{N}  \jmath^*\Big(\Alt_m(\Omega_{12})\Big)\,.
\label{Eq:omega-Alt-Omega}
%%%%%%%%%%%%%%%%%%%%%%%%%%%%%%%%%%%%%%%%
\end{align}%%%%%%%%%%%%%%%%%%%%%%%%%%%%%%%%%%
%%%%%%%%%%%%%%%%%%%%%%%%%%%%%%%%%%%%%%%%
%\begin{align}
%\omega_B
%&=
%\frac{N-2}{N!}
%\sum_{\sigma\in\mathfrak S_m}
%(-1)^\sigma\sigma^*\big( \Omega_{12}\big)
%=
%\frac{N-2}{N}\Alt_m(\Omega_{12}).
%\label{Eq:omega-Alt-Omega}
%\end{align}
Since $N-2=n$ and $m=n+1$, Lemma~\ref{L:Alt-Omega} yields the following
comparison formula.

%%%%%%%%%%%%%%%%%%%%%%%%%%%%%%%%%%%%%%%%%
\begin{prop}
\label{P:comparison-explicit}
%%%%%%%%%%%%%%%%%%%%%%%%%%%%%%%%%%%%%%%%%
For every even integer $n\geq2$, on the affine chart of $\M_{0,n+2}$
described above, one has
%\begin{equation}
%\label{Eq:omega-epsilon}
%\boxed{
%\omega_B=
%-\frac{2}{(n+1)(n+2)}\,\varepsilon_{n-1}(x')\,
%\dd x_2\wedge\cdots\wedge\dd x_n.}
%\end{equation}
%Equivalently,
\begin{equation}
\label{Eq:omega-epsilon}
\varepsilon_{n-1}(x')\,
\dd x_2\wedge\cdots\wedge\dd x_n
=
\frac{(n+1)(n+2)}{2}\,\omega_B\,.
\end{equation}
%Equivalently,
%%%%%%%%%%%%%%%%%%%%%%%%%%%%%%%%%%%%%%%%%
\end{prop}%%%%%%%%%%%%%%%%%%%%%%%%%%%%%%%%%%%
%%%%%%%%%%%%%%%%%%%%%%%%%%%%%%%%%%%%%%%%%
\begin{proof}
Combining \eqref{Eq:omega-Alt-Omega} with \eqref{Eq:Alt-Omega}, we obtain
\[
\omega_B
=
\frac{n}{n+2}
\left(
\frac{2}{n(n+1)}
\varepsilon_{n-1}(x')\,\boldsymbol\eta
\right)
=
\frac{2}{(n+1)(n+2)}
\varepsilon_{n-1}(x')\,\boldsymbol\eta,
\]
which is equivalent to \eqref{Eq:omega-epsilon}.
\end{proof}

%\textcolor{red}{\footnote{OLD}Formula \eqref{Eq:omega-epsilon} proves, for every even $n\geq2$, the
%explicit formula conjectured in \cite[\S5.9.2]{M06}: the form
%$\varepsilon_{n-1}(x')\,\dd x_2\wedge\cdots\wedge\dd x_n$ occurring there
%is, up to the explicit nonzero scalar $-(n+1)(n+2)/2$, exactly the form
%$\omega_B$ constructed in the present paper. Since the other components are
%obtained by relabelling, the same scalar relates the corresponding complete
%abelian relations.}
%\bk

Formula~\eqref{Eq:omega-epsilon} proves, for every even $n\geq2$, the
explicit formula conjectured in \cite[\S5.9.2]{M06}: the form
$\varepsilon_{n-1}(x')\,\dd x_2\wedge\cdots\wedge\dd x_n$ occurring there
is, up to the explicit nonzero scalar $(n+1)(n+2)/2$, exactly the form
$\omega_B$ constructed in the present paper.  Let $\boldsymbol{\mathcal E}'_n$ be the abelian relation of \cite[(5.64)]{M06},
whose $i$th component is
$\big(C^{\circ(i-1)}\big)^*\big(\varepsilon_{n-1}(x')\,\dd x_2\wedge\cdots\wedge\dd x_n\big)$
for $i=1,\ldots,n+3$, where $C$ is the cyclic birational map of
\cite[(4.3)]{M06} and $\varepsilon_{n-1}$ is given by \eqref{Eq:epsilon-M06}.

In the coordinates of \cite[\S5.9.2]{M06}, the points $x_1,\ldots,x_n,0,1,\infty$
of $\mathbf P^1$ 
carry the labels $1,\ldots,n+3$ respectively; this is the labelling of the
foliations used in \cite[\S4.1 and \S5.9]{M06}, which differs from the labelling
$[0,1,\infty,x_1,\ldots,x_n]$ of the points in \cite[(2.1)]{M06}. The first component is
%In the coordinates of \cite[\S5.9.2]{M06}, the points $x_1,\ldots,x_n,0,1,\infty$
%of $\mathbf P^1$ 
%carry the labels $1,\ldots,n+3$ respectively, and the first component is
 pulled back by the forgetful map $f_1$, so that the affine chart of
Proposition~\ref{P:comparison-explicit} is the natural chart of $\M_{0,A_1}$.
Since the component of $\boldsymbol{E}_n$ indexed by $1$ is $-f_1^*(\omega_1)$,
\eqref{Eq:omega-epsilon} gives
$$
f_1^*\big(\varepsilon_{n-1}(x')\,\dd x_2\wedge\cdots\wedge\dd x_n\big)
=\frac{(n+1)(n+2)}{2}\,f_1^*(\omega_1)
=-\frac{(n+1)(n+2)}{2}\,\big(\boldsymbol{E}_n\big)_1 .
$$

Moreover, by the lemma of \cite[\S4.1]{M06}, one has $C^*(\mathcal F_i)=\mathcal F_{c(i)}$
for every $i$, where $\mathcal F_i$ is the foliation defined by $f_i$ and
$c=(1\,2\,\cdots\,n{+}3)$. Since the automorphism group of $\M_{0,A}$ acts faithfully on the
foliations of $\boldsymbol{\mathcal W}_{0,A}$, $C$ is the relabelling automorphism associated
with $c$. Since $n+3$ is odd, $c$ is an even permutation, and
 the second point of Theorem~\ref{Prop:Differential-Identity} gives
$\big(c^{\,i-1}\big)^*\big(\boldsymbol{E}_n\big)_1=\big(\boldsymbol{E}_n\big)_i$
for $i=1,\ldots,n+3$. Therefore
$$
\boldsymbol{\mathcal E}'_n=-\frac{(n+1)(n+2)}{2}\,\boldsymbol{E}_n\, ,
$$
which proves Corollary~\ref{Cor:comparison}.

%%%%%%%%%%%%%%%%%%%%%%%%%%%%%%%%%%%%%%%%%
%%%%%%%%%%%%%%%%%%%%%%%%%%%%%%%%%%%%%%%%%
%%%%%%%%%%%%%%%%%%%%%%%%%%%%%%%%%%%%%%%%%
%% ICI ICI ICI 
%%
%%%%%%%%%%%%%%%%%%%%%%%%%%%%%%%%%%%%%%%%%
%%%%%%%%%%%%%%%%%%%%%%%%%%%%%%%%%%%%%%%%%

\begin{cor}[Second point of Proposition~\ref{Prop:descent}]
\label{Cor:Second-Point}
For every ordered set $A$ of even cardinality $N\geq 4$,  $\omega_A$ is 
a real-analytic form which is not rational on any nonempty open subset of
$\boldsymbol{\mathcal M}_{0,A}^{\mathbb R}$. In particular, $\omega_A$ is non-zero.
\end{cor}
\begin{proof}
By relabelling, we may assume $A=B$ and work in the chart of \S\ref{SS:Identification}, with $n=N-2$.
By Lemma \ref{L:epsilon-C} and Equation \eqref{Eq:omega-epsilon}, it suffices to prove that 
$$ \sum_{k<\ell}K_{k\ell}(y)\,\mathrm{Log}\lvert \,y_{k\ell}\, \lvert $$ 
is not rational on any nonempty open subset. Since all the $K_{k\ell}$ are nonzero rational
functions, and since a real-analytic identity on a nonempty open subset extends by analytic
continuation, this follows  from the fact that the constant function $1$ and the nonzero functions among the $\mathrm{Log}\lvert y_{k\ell}\rvert$'s 
are linearly independent over the field  of rational
functions in the $y_i$'s (an immediate extension of \cite[Lemma~5.9.6(3)]{M06}).
\end{proof}

%%%%%%%%%%%%%%%%%%%%%%%%%%%%%%%%%%%%%%%%%%%%%%%
%%%%%%%%%%%%%%%%%%%%%%%%%%%%%%%%%%%%%%%%%%%%%%%
%\section{}%%%%%%%%%%%%%%%%%%%%%%%%%%%%%%%%%%%%%%%%%%
%%%%%%%%%%%%%%%%%%%%%%%%%%%%%%%%%%%%%%%%%%%%%%%
%%%%%%%%%%%%%%%%%%%%%%%%%%%%%%%%%%%%%%%%%%%%%%%

%%%%%%%%%%%%%%%%%%%%%%%%%%%%%%%%%%%%%%%

%%%%%%%%%%%%%%%%%%%%%%%%%%%%%%%%%%%%%%%
%%%%%%%%%%%%%%%%%%%%%%%%%%%%%%%%%%%%%%%

\end{document}